\documentclass{amsart}
\usepackage{cite}
\usepackage[all]{xy}
\usepackage{graphicx}
\usepackage{textcomp}
\usepackage{charter}
\usepackage[vflt]{floatflt}
\usepackage{amssymb}
\usepackage{mathabx}
\usepackage{enumerate}
\usepackage[pdftex,colorlinks]{hyperref}
\usepackage{mathrsfs}
\hypersetup{
pdfauthor={Shamovich, E.},
bookmarksnumbered,
pdfstartview={FitH}}%

\theoremstyle{plain}
\newtheorem{thm}{Theorem}[section]
\newtheorem{cor}[thm]{Corollary}
\newtheorem{lem}[thm]{Lemma}
\newtheorem{prop}[thm]{Proposition}

\theoremstyle{definition}
\newtheorem{dfn}[thm]{Definition}

\theoremstyle{remark}
\newtheorem{rem}[thm]{Remark}

\numberwithin{equation}{section}

\newcommand{\pd}[2]
{\dfrac{\partial #1}{\partial #2}}

\newcommand{\GL}{\operatorname{GL}}

\newcommand{\fV}{{\mathfrak V}}

\newcommand{\fu}{{\mathfrak u}}
\newcommand{\fy}{{\mathfrak y}}

\newcommand{\RR}{{\mathbb R}}

\newcommand{\cA}{{\mathcal A}}
\newcommand{\cE}{{\mathcal E}}
\newcommand{\cH}{{\mathcal H}}

\newcommand{\cS}{{\mathcal S}}
\newcommand{\cL}{{\mathcal L}}

\newcommand{\cM}{{\mathcal M}}

\newcommand{\cK}{{\mathcal K}}
\newcommand{\cF}{{\mathcal F}}

\newcommand{\cW}{{\mathcal W}}
\newcommand{\cX}{{\mathcal X}}
\newcommand{\cY}{{\mathcal Y}}

\newcommand{\C}{{\mathbb C}}
\newcommand{\R}{{\mathbb R}}

\newcommand{\N}{{\mathbb N}}

\newcommand{\Span}{\operatorname{Span}}

\newcommand{\im}{\operatorname{Im}}
\newcommand{\tempdist}{\cS^{\prime}(\R,\cE)}
\newcommand{\tempdistd}{\cS^{\prime}(\R^d,\cE)}
\newcommand{\schwartz}{\cS(\R,\cE)}
\newcommand{\schwartzd}{\cS(\R^d,\cE)}
\newcommand{\Ltwoe}{L^2(\R,\cE)}
\newcommand{\Ltwoex}{L^2(\R,\cE,\alpha(x))}
\newcommand{\Ltwoej}{L^2(\R,\cE,\alpha_j)}
\newcommand{\Pos}{\operatorname{Pos}}

\newcommand{\vr}{very reasonable}

\begin{document}

\title{Dilations of Semigroups of Contractions through Vessels}
\author{Eli Shamovich}
\address{Department of  Mathematics\\ 
Technion - Israel Institute of Mathematics\\
Haifa, 3200003, Israel}
\email{shamovich@tx.technion.ac.il}

\author{Victor Vinnikov}
\address{Department of  Mathematics\\ 
Ben-Gurion University of the Negev\\
84105 Beer-Sheva, Israel}
\email{vinnikov@math.bgu.ac.il}

\thanks{This paper is partially based on the results appearing in the Ph. D. thesis of E. S. written under the supervision of V. V. in the Ben-Gurion university of the Negev. Both authors were partially supported by
US--Israel  BSF grant 2010432.}
\begin{abstract}
Let $A_1,\ldots,A_d$ be a $d$-tuple of commuting dissipative operators on a separable Hilbert space $\cH$. Using the theory of operator vessels and their associated systems, we give a construction of a dilation of the multi-parameter semigroup of contractions on $\cH$ given by $e^{i \sum_{j=1}^d t_j A_j}$.
\end{abstract}
\maketitle
\tableofcontents

\section{Introduction} \label{sec:intro}

The question of dilating a contraction to a unitary was considered first in 1953 by B. Sz.-Nagy in \cite{SNa53} (P. Halmos constructed a dilation in \cite{Hal50}, but it was not a power dilation); the first explicit construction of a unitary dilation was given by J. Sch\"{a}ffer in \cite{Sch55}. In \cite{SNa53} Sz.-Nagy has also constructed dilations of one-parameter semigroups of contractions. Nagy's dilation theorem was extended to pairs of commuting contractions by T. Ando in \cite{An63}. However, in 1970 S. Parrot provided a counterexample to the existence of commuting unitary dilations for three commuting contractions in \cite{Par70}. More examples appeared later, see for example \cite{Var74}. In \cite{Arv69} and \cite{Arv72} W. Arveson generalized dilation theory to the setting of arbitrary operator algebras and their representations. More historical background and details are available in \cite{Arv10} and \cite{NFBK10}.

Our goal in this paper is to give an explicit construction of commutative unitary dilations of certain multi-parameter commutative semigroups of contractions. To be more precise, given a $d$-tuple of commuting dissipative operators on a separable Hilbert space $\cH$, we consider the multiparameter semigroup $S$ of contractions that they generate; we provide conditions on $A_1,\ldots,A_d$, such that $S$ admits a dilation to a commutative group of unitaries. 

As was already implicit in the works of M. S. Livsic (see \cite{BrLiv60}) and P. Lax and R. Phillips (\cite{LaxPhl89}) and became explicit later (\cite{AdAr70}, \cite{BallCoh91}, \cite{Hel74} and \cite{CBMS87}), the construction of a unitary dilation has a simple system-theoretic interpretation: we embed the contraction into a conservative discrete-time input/state/output (i/s/o) linear system and consider the Hilbert space of square-summable trajectories with the natural shift operator. See also Sarason's Lemma \cite[Lem.\ 0]{Sar65} that shows that any unitary dilation is obtained in this way, the works of B. Pavlov \cite{Pav77} and \cite{Pav99}, and \cite{BSV04} for a survey of various mutlidimensional cases. We review these ideas in Section \ref{sec:dilations} (see Lemma \ref{lem:one_dim_ext} and Proposition \ref{prop:one_dim_dilation}). Motivated by this we construct an overdetermined multidimensional conservative (continuous-time) i/s/o system and consider a Hilbert space of certain trajectories of this system with a natural unitary representation of $\R^d$ on it by shifts. We expect that any commutative unitary dilation of a commutative semigroup of contractions arises in this way, so that the sufficient conditions on the operators $A_1,\ldots,A_d$, that we describe, are necessary; we plan to address this question in a future work. For other sufficient conditions for the existence of a commutative unitary dilation in the discrete-time case, see \cite{An76}, \cite{Arch07}, \cite{Breh61}, \cite{GaSu97}, \cite{GaSu01} and \cite{GKVVW09}.

In Section \ref{sec:vessels} we introduce the main tool in the construction, namely commutative operator vessels and their associated systems. More concretely, if $A_1,\ldots,A_d$ is a $d$-tuple of commuting bounded operators on a separable Hilbert  space $\cH$, we associate to them a collection of spaces and operators:
\[
\fV = \left(\cH,\cE,\Phi,\{A_j\}_{j=1}^d,\{\sigma_j\}_{j=1}^d,\{\gamma_{jk}\}_{j,k=1}^d,\{\gamma_{*jk}\}_{j,k=1}^d \right)
\]
Here $\cE$ is an auxiliary separable Hilbert space, $\Phi \colon \cH \to \cE$ is a bounded operator and for every $j,k = 1,\ldots,d$, $\sigma_j$, $\gamma_{jk}$ and $\gamma_{*jk}$ are selfadjoint bounded operators on $\cE$ satisfying some conditions described in detail in the text (see \eqref{eq:colligation_condition} and \eqref{eq:vessel_conditions}). The study of operator vessels was initiated by M. S. Livsic (see for example the papers \cite{BrLiv60}, \cite{LivVak74}, \cite{Liv78} and \cite{Liv79wav} and the book \cite{LKMV}). A functional model for two commuting dissipative operators with finite-dimensional imaginary parts was constructed by J. Ball and the second author in \cite{BV} using frequency domain methods. We briefly recall the relevant notions of the associated overdetermined multidimensional system, the adjoint system and the input and output compatibility conditions in both the continuous and the discrete-time setting. We introduce all the necessary background, notions and results. Some of the results are proved for the sake of completeness. In particular we show that there is a natural way, given a $d$-tuple of commuting operators, to embed them in a so-called strict vessel.

In Section \ref{sec:discrete_solve} we consider the system of input and output compatibility conditions in the analytic case. If $d \geq 3$ this system is itself overdetermined. We find therefore necessary and sufficient conditions on the vessel, so that the system of input (or output) compatibility conditions admits a solution for any initial condition along one of the axes analytic in some neighborhood of the origin. We call these conditions very reasonable conditions or $VR$ for short. We then show in Section \ref{sec:vr} that the $VR$ conditions are independent of the choice of the axis and they hold at the input if and only if they hold at the output. We also show that in the case of a doubly commuting $d$-tuple of operators, we have the $VR$ conditions automatically for the strict vessel embedding.

We proceed in Section \ref{sec:cont_solve} to show that if the $VR$ conditions hold and the system of continuous-time compatibility conditions is hyperbolic, then it has a weak solution in tempered distributions for every initial condition along one of the axes. We proceed to show that under certain assumptions the distribution is in fact a function that is in $L^2$ on lines with respect to a twisted inner product. We write a transform taking the initial condition along the $t_1$ axis into a condition along the $t_j$ axis and demonstrate some of its properties. This section forms the technical toolbox for the proof of the main dilation theorem.

In Section \ref{sec:main} we state the main dilation theorem: if a $d$-tuple of commuting dissipative operators $A_1,\ldots,A_d$ possesses the dissipative embedding property, namely if they can be embedded into a vessel satisfying the $VR$ conditions and such that $\sigma_1, \ldots, \sigma_d \geq 0$, then the semigroup of contraction generated by $A_1,\ldots,A_d$ admits a dilation to a commutative group of unitaries. In particular since the $VR$ conditions are vacuous when $d = 2$ we obtain a version of Ando's dilation theorem for the continuous-time case. There are some technical restrictions, since we are dealing with bounded generators, however our construction of the unitary dilation is completely explicit and thus should allow a further geometric analysis, similarly to the one dimensional case. 

We prove the main dilation theorem, i.e., we construct the dilation space and the group of unitaries in Section \ref{sec:dilations} using the tools developed in Section \ref{sec:cont_solve}. We conclude by demonstrating a necessary and a sufficient condition for the dilation thus obtained to be minimal.

\section{Livsic Commutative Operator Vessels} \label{sec:vessels}

In this section we recall briefly the notion of Livsic commutative operator vessels, for more information see \cite{LKMV}.

\begin{dfn}
Let $\cH$ be a separable Hilbert space and $A_1, \ldots, A_d$ a $d$-tuple of commuting non-selfadjoint bounded operators on $\cH$. We fix an auxiliary separable Hilbert space $\cE$, a bounded operator $\Phi \colon \cH \to \cE$ and a $d$-tuple of bounded selfadjoint operators on $\cE$, $\sigma_1,\ldots,\sigma_d$, satisfying the colligation condition, namely for every $k = 1,\ldots,d$:
\begin{equation} \label{eq:colligation_condition}
A_k - A_k^* = i \Phi^* \sigma_k \Phi.
\end{equation}
We also fix two collections of bounded selfadjoint operators, $\gamma_{jk}$ and $\gamma_{*jk}$, for $j,k = 1,\ldots,d$ on $\cE$, satisfying the following set of conditions:
\begin{align} \label{eq:vessel_conditions}
\begin{split}
& \bullet \gamma_{jk} = -\gamma_{kj}, \gamma_{*jk} = -\gamma_{*kj},\\
& \bullet \sigma_j \Phi A_k^* - \sigma_k \Phi A_j^* = \gamma_{jk} \Phi, \\
& \bullet \sigma_j \Phi A_k - \sigma_k \Phi A_j = \gamma_{*jk} \Phi, \\
& \bullet \gamma_{*jk} - \gamma_{jk} = i \left(\sigma_j \Phi \Phi^* \sigma_k - \sigma_k \Phi \Phi^* \sigma_j \right).
\end{split}
\end{align}
The collection of operators and spaces satisfying the above conditions is called a Livsic commutative operator vessel.
\end{dfn}

Given a Livsic commutative vessel one can associate to it an energy preserving linear time invariant overdetermined system in continuous-time. Let $u,y \colon \R^d \to \cE$ be smooth functions, we call them the input and output signals, respectively, and let $x \colon \R^d \to \cH$ be a smooth function that we call the state, then we define the system:
\begin{align*}
\begin{split}
& i \pd{x}{t_k} + A_k x = \Phi^* \sigma_k u, \\
& y = u - i \Phi x.
\end{split}
\end{align*}
We will assume for now that $u,y \in C^1(\R^d,\cE)$ and $x \in C^2(\R^d,\cH)$. We will discuss latter on various relaxations of this assumption.

For $d > 1$ the above system is overdetermined and hence requires input and output compatibility conditions. It follows from the vessel conditions \eqref{eq:vessel_conditions} (for details see \cite[Thm.\ 3.2.1]{LKMV}) that the necessary and sufficient input compatibility conditions are given by:
\begin{equation} \label{eq:ns_input_compat}
\Phi^*\left(\sigma_k \pd{u}{t_j} - \sigma_j \pd{u}{t_k} + i \gamma_{jk} u\right) = 0.
\end{equation}
We define the strict input compatibility conditions by:
\begin{equation} \label{eq:input_compat}
\sigma_k \pd{u}{t_j} - \sigma_j \pd{u}{t_k} + i \gamma_{jk} u = 0.
\end{equation}
Similarly at the output we get the following system of compatibility conditions and strict compatibility conditions:
\begin{equation} \label{eq:ns_output_compat}
\Phi^*\left(\sigma_k \pd{y}{t_j} - \sigma_j \pd{y}{t_k} + i \gamma_{*jk} y\right) = 0.
\end{equation}
\begin{equation} \label{eq:output_compat}
\sigma_k \pd{y}{t_j} - \sigma_j \pd{y}{t_k} + i \gamma_{*jk} y = 0.
\end{equation}
When $d > 2$ we note that the system of input compatibility conditions \eqref{eq:input_compat} is itself overdetermined. The goal of the current paper is to understand the additional compatibility conditions on \eqref{eq:input_compat} required for the system to have "enough" solutions in the hyperbolic case, i.e., when the operators $A_1,\ldots,A_d$ are dissipative and $\sigma_1,\ldots,\sigma_d \geq 0$. We then use these solutions to construct a unitary dilation for the semigroup of contractions generated by $A_1,\ldots,A_d$.

One can show that if $u$ solves the system of input compatibility conditions, then for each initial condition $x(0) = h \in \cH$, there exists a unique state $x$ solving the system and the output $y$ then satisfies the output compatibility conditions, see \cite{BV} for the $d=2$ case and \cite{LKMV,SV14a} for the general case. The formula for $x$ is then:
\begin{multline} \label{eq:mult-dim_x}
x(t_1,\ldots,t_d) = e^{i \sum_{j=1}^d t_j A_j} \left( h - \right. \\ i \left. \int_0^{(t_1,\ldots,t_d)} e^{ -i \sum_{j=1}^d s_j A_j} \Phi^* \left(\sum_{j=1}^d \sigma_j u(s_1,\ldots,s_d) ds_j \right) \right).
\end{multline}
We also have the adjoint vessel:
\[
\fV^* = \left(\cH,\cE,-\Phi,\{A_j^*\}_{j=1}^d,\{-\sigma_j\}_{j=1}^d,\{-\gamma_{jk}\}_{j,k=1}^d,\{-\gamma_{*jk}\}_{j,k=1}^d \right)
\]
The adjoint system, namely the associated system of the adjoint vessel, is given by:
\begin{align} \label{eq:adjoint_system}
\begin{split}
& i \pd{\tilde{x}}{t_k} + A_k^* \tilde{x} = \Phi^* \sigma_k \tilde{u}, \\
& \tilde{y} = \tilde{u} + i \Phi \tilde{x}. 
\end{split}
\end{align}

It is proved in \cite{BV}, \cite{LKMV} and \cite{SV14a} that $(u,x,y)$ is a system trajectory for the associated system of $\fV$ if and only if $(y,x,u)$ is a system trajectory for the adjoint system. Using this we deduce the energy balance equations (\cite[Cor.\ 1.2.8]{SV14a}). Namely for a trajectory $(u,x,y)$ of the associated system we have for every $t = (t_1,\ldots,t_d) \in \R^d$:
\begin{multline} \label{eq:energy_balance}
\|x(t + s e_j)\|^2 - \|x(t)\|^2 = \\ \int_{0}^{s} \langle \sigma_j u(t+ p e_j) ,u(t + p e_j) \rangle dp - \int_0^s \langle \sigma_j y(t + p e_j), y(t + p e_j) \rangle dp.
\end{multline}
We briefly recall the proof for the sake of completeness:
\begin{multline*}
\pd{}{t_j} \langle x , x \rangle = \langle \pd{x}{t_j} , x \rangle + \langle x , \pd{x}{t_j} \rangle = \langle i A_j x - i \Phi^* \sigma_j u , x \rangle + \langle x , i A_j x - i \Phi^* \sigma_j u \rangle = \\
\langle i (A_j - A_j^*) x , x \rangle - i \langle u, \sigma_j \Phi x \rangle + i \langle \sigma_j \Phi x , u \rangle = \\ \langle u , \sigma_j ( u - y) \rangle + \langle \sigma_j (u - y) , u \rangle - \langle \sigma_j \Phi x , \Phi x \rangle = \langle \sigma_j u , u \rangle - \langle \sigma_j y , y \rangle.
\end{multline*}
Now all that remains is to integrate with respect to $t_j$ to get the desired result. Note that it follows from this proof that if $(u,x,y)$ is a system trajectory, with $u$ and $y$ locally integrable on every line parallel to one of the axes and such that the system still admits a solution $x$ that is absolutely continuous on every such line and thus almost everywhere differentiable on it (as a function of one variable), then the energy balance equations still hold. We will use this comment in the following sections.

\begin{dfn} \label{dfn:strict_vessel}
We will say that a vessel $\fV$ is strict if:
\begin{itemize}
 \item $\Phi$ is surjective,
 \item $\cap_{j=1}^d \ker \sigma_j = 0$.
\end{itemize} 
We will say that $\fV$ is weakly strict if $\cap_{j=1}^d \ker \Phi^* \sigma_j = 0$.
\end{dfn}

Clearly, if $\fV$ is strict it is weakly strict; the converse is not necessarily true. It was shown in \cite{LKMV} that every $d$-tuple of commuting operators admits an embedding into an essentially unique strict vessel (see \cite{SV14a} for a non-commutative case). The embedding is given as follows:
\begin{align} \label{eq:strict_vessel}
\begin{split}
& \cE = \overline{\sum_{j=1}^d \im(A_j - A_j^*)}, \\
& \Phi = P_{\cE}, \\
& \sigma_j = \frac{1}{i}(A_j - A_j^*)|_{\cE}, \\
& \gamma_{jk} = \frac{1}{i}(A_j A_k^* - A_k A_j^*)_{\cE},\\
& \gamma_{*jk} = \frac{1}{i}(A_j^*A_k - A_k^*A_j)|_{\cE}.
\end{split}
\end{align}
The subspace $\cE$ is called the non-Hermitian subspace of the $d$-tuple $A_1,\ldots,A_d$.

\section{Solution in the Analytic Case} \label{sec:discrete_solve}

Assume that $u$ is a real analytic function with a convergent power series expansion around the origin:
\[
u(t) = \sum_{n \in \N^d} a(n) t^n.
\]
Then plugging the power series into \eqref{eq:input_compat} we get the following difference equation on the coefficients of the power series:
\[
\sigma_k a(n + e_j) - \sigma_j a(n + e_k) + i \gamma_{jk} a(n) = 0.
\]

\begin{rem}
Consider a collection of operators and spaces satisfying the vessel conditions \eqref{eq:vessel_conditions}, i.e., a vessel in the more general sense of \cite{BV} and \cite[Part III]{LKMV} (without the colligation condition \eqref{eq:colligation_condition}). We can associate to such a general vessel a linear overdetermined discrete-time i/s/o system, such that the corresponding system of input compatibility conditions is the system of difference equations obtained above.
\end{rem}

We will assume from now on that $\sigma_1$ is invertible. Denote by $d_j$ the shift operator in the $j$-th coordinate, namely $(d_j a)(n) = a(n + e_j)$, for every $n \in \N^d$. We would like to be able to solve the system for every initial condition along the $n_1$ axis. Consider an equation that contains $\sigma_1$, and multiply it by $\sigma_1^{-1}$ to get:
\begin{equation} \label{eq:k-shift}
a(n + e_k) = \sigma_1^{-1} \sigma_k a(n + e_1) + i \sigma_1^{-1} \gamma_{1k} a(n).
\end{equation}
Applying first $d_j$ and then $d_k$ we get:
\begin{multline*}
a(n + e_k + e_j) = \sigma_1^{-1} \sigma_k a(n + e_j + e_1) + i \sigma_1^{-1} \gamma_{1k} a(n + e_j) = \\ \sigma_1^{-1} \sigma_k \sigma_1^{-1} \sigma_j a( n + 2 e_1) + i \sigma_1^{-1} \sigma_k  \sigma_1^{-1} \gamma_{1j} a(n + e_1) + i \sigma_1^{-1} \gamma_{1k} \sigma_1^{-1} \sigma_j a(n + e_1) - \\ \sigma_1^{-1} \gamma_{1k} \sigma_1^{-1} \gamma_{1j} a(n).
\end{multline*}
Since the shift operators along different axes commute we get that:
\begin{multline*}
\left( \sigma_k \sigma_1^{-1} \sigma_j - \sigma_k \sigma_1^{-1} \sigma_j \right) a( n + 2 e_1) + \left(\gamma_{1j} \sigma_1^{-1} \gamma_{1k} -  \gamma_{1k} \sigma_1^{-1} \gamma_{1j}  \right)a(n) + \\ i \left( \sigma_k  \sigma_1^{-1} \gamma_{1j} + \gamma_{1k} \sigma_1^{-1} \sigma_j - \sigma_j  \sigma_1^{-1} \gamma_{1k} -  \gamma_{1j} \sigma_1^{-1} \sigma_k  \right) a(n + e_1) = 0.
\end{multline*}
Now if we take $n = 0$ and use the fact that we require the system to be solvable for every initial condition along the $n_1$-axis we get the following set of necessary conditions:
\begin{align} \label{eq:commutation_conditions}
\begin{split}
& \bullet [\sigma_1^{-1} \sigma_j, \sigma_1^{-1} \sigma_k] = 0, \\
& \bullet [\sigma_1^{-1} \gamma_{1j}, \sigma_1^{-1} \gamma_{1k}] = 0, \\
& \bullet [\sigma_1^{-1} \sigma_k, \sigma_1^{-1} \gamma_{1j}] = [\sigma_1^{-1} \sigma_j, \sigma_1^{-1} \gamma_{1k}].
\end{split}
\end{align}
There are more necessary conditions, since we have a lot of equations that do not involve $\sigma_1$. We take such an equation and use \eqref{eq:k-shift} to get:
\begin{multline*}
\sigma_k \sigma_1^{-1} \sigma_j a(n + e_1) + i \sigma_k \sigma_1^{-1} \gamma_{1j} a(n) - \sigma_j \sigma_1^{-1} \sigma_k a(n + e_1) - \\ i \sigma_j \sigma_1^{-1} \gamma_{1k} a(n) + i \gamma_{jk} a(n) = 0.
\end{multline*}
If we apply now the necessary conditions \eqref{eq:commutation_conditions} and use again the fact that for $n = 0$ the vector $a(0)$ is arbitrary we get that:
\begin{equation} \label{eq:additional_condition}
\gamma_{jk} = \sigma_j \sigma_1^{-1} \gamma_{1k} - \sigma_k \sigma_1^{-1} \gamma_{1j}.
\end{equation}
Then the following proposition is almost immediate:
\begin{prop} \label{prop:conditions_discrete}
The conditions \eqref{eq:commutation_conditions} and \eqref{eq:additional_condition} are necessary and sufficient for the system of discrete-time input compatibility equations to have a solution for every initial condition along the $n_1$-axis.
\end{prop}
\begin{proof}
We have seen above that this set of conditions is necessary. Now for sufficiency note that using \eqref{eq:additional_condition} we can eliminate all of the equations that do not involve $\sigma_1$. The other equations are compatible by \eqref{eq:commutation_conditions} and thus for every initial $a(n_1,0,\ldots,0) = b(n_1)$ they define a unique function $a$. In fact if $a(n_1,0,\ldots,0) = b(n_1)$ we can write $a$ as follows (recall that $d_1$ stands for the shift in the first coordinate):
\begin{equation} \label{eq:discret_solution}
a(n_1,\ldots,n_d) = (\alpha_2 d_1 + i \beta_2)^{n_2} \cdots (\alpha_d d_1 + i \beta_d)^{n_d} b(n_1).
\end{equation}
Here $\alpha_j = \sigma_1^{-1} \sigma_j$ and $\beta_j = \sigma_1^{-1} \gamma_{1j}$. 
\end{proof}

\begin{rem}
For the necessity part of Proposition \ref{prop:conditions_discrete} it is in fact enough to assume that $a(0)$, $a(e_1)$ and $a(2 e_1)$ are arbitrary.
\end{rem}

\begin{dfn}
We will call a vessel that satisfies conditions \eqref{eq:commutation_conditions} and \eqref{eq:additional_condition} \vr{}.
\end{dfn}

\begin{rem}
In \cite[Cor.\ 2.20]{SV14} and the following discussion, similar conditions were given for a tensor $\gamma \in M_n(\C) \otimes \wedge^{k+1} \C^{d+1}$ to be \vr{}. The difference is that in \cite{SV14} we require $\cE$ to be finite-dimensional and we also require generic semisimplicity, whereas in the case at hand we do not need either.
\end{rem}

We now use \eqref{eq:commutation_conditions} and \eqref{eq:additional_condition} to describe the solution in the continuous-time case when the initial condition is an $\cE$-valued analytic function in a neighbourhood of $0$. Recall, that a function $f \colon (-r,r) \to \cE$ is strongly analytic at $0$ if there exist $\{\xi_n\}_{n=0}^{\infty} \subset \cE$, such that $f(t) = \sum_{n=0}^{\infty} \xi_n t^n$, where the series converges in norm for every $t$ in a neighborhood of $0$. This condition is in fact equivalent to weak analyticity, namely that for every $\xi \in \cE$ the function $\langle f(t), \xi \rangle$ is real analytic in a neighbourhood of $0$. The following theorem can be thought of as a version of the classical Cauchy-Kowalevskaya theorem (cf. \cite[Sec.\ I.D]{Fol95}).

\begin{thm} \label{thm:compatibility_solutions_analytic}
Assume that we are given an initial condition $u(t_1,0,\ldots,0) = f(t_1)$ analytic near the origin and that \eqref{eq:commutation_conditions} and \eqref{eq:additional_condition} are satisfied, then there exists an open neighborhood of the origin and a unique analytic solution $u$ to the input compatibility system.
\end{thm}
\begin{proof}
First note that as in the discrete-time case, \eqref{eq:additional_condition} allows us to eliminate all of the equations that do not involve $\sigma_1$, hence we are left with the system ($j = 2,\ldots,d$):
\[
\pd{u}{t_j} = \alpha_j \pd{u}{t_1} + \beta_j u.
\]
Here $\alpha_j = \sigma_1^{-1} \sigma_j$ and $\beta_j = i \sigma_1^{-1} \gamma_{1j}$. We write an expansion for the initial condition $f$ and the solution $u$ in a polydisc around the origin and solve for the coefficients. This reduces the problem to the discrete-time case that we have already seen, except that we have to verify that the series for $u$ is locally convergent. If the radius of convergence of the series of $f(t_1) = \sum_{m=0}^{\infty} b(m) t_1^m$ around $0$ is $R$, then for every $0 < r < R$, there exists a constant $M > 0$, such that $\|b(m)\| \leq \frac{M}{r^m}$. Now we note that we can obtain the coefficients of $u$ in terms of the $b(m)$ using Equation \ref{eq:discret_solution}. If we set $C = \max \{ \|\alpha_j\|,\|\beta_j\| \mid j = 2, \ldots d \}$, then we get that:
\[
\|a(n)\| \leq C^{|n| - n_1} M \frac{1}{r^{n_1}} (1 + \frac{1}{r})^{|n| - n_1}.
\]
Therefore, the series for $u$ will converge in the closed polydisc around the origin with polyradius $(r, \frac{r}{C(r+1)}, \ldots, \frac{r}{C(r+1)})$. Since this is true for every $0 < r < R$, we see that the series for $u$ will converge in the polydisc around the origin with polyradius $(R, \frac{R}{C(R+1)}, \ldots, \frac{R}{C(R+1)})$.

\end{proof}

\section{Very Reasonable Conditions} \label{sec:vr}

In this Section we will discuss the conditions \eqref{eq:commutation_conditions} and \eqref{eq:additional_condition} that we will call very reasonable conditions or $VR$ conditions for brevity. We have defined the $VR$ conditions in the case of the input compatibility system. A similar set of conditions arises at the output, namely for $j,k = 2,\ldots,d$:
\begin{align} \label{eq:output_VR}
\begin{split}
& \bullet [\sigma_1^{-1} \sigma_j, \sigma_1^{-1} \sigma_k] = 0, \\
& \bullet [\sigma_1^{-1} \gamma_{*1j}, \sigma_1^{-1} \gamma_{*1k}] = 0, \\
& \bullet [\sigma_1^{-1} \sigma_k, \sigma_1^{-1} \gamma_{*1j}] = [\sigma_1^{-1} \sigma_j, \sigma_1^{-1} \gamma_{*1k}],\\
& \bullet \gamma_{*jk} = \sigma_j \sigma_1^{-1} \gamma_{*1k} - \sigma_k \sigma_1^{-1} \gamma_{*1j}.
\end{split}
\end{align}
Let us call this system of conditions $VR_*$ conditions. We will now investigate the relation between the $VR$ and $VR_*$ conditions. 

\begin{prop} \label{prop:VR_iff_VR_*}
Given a vessel $\fV$, it satisfies the $VR$ conditions if and only if it satisfies the $VR_*$ conditions.
\end{prop}
\begin{proof}
It suffices to prove only one implication, since the proof of the other will be symmetric. Assume that the $VR$ conditions hold and thus the first condition of the $VR_*$ conditions is automatically satisfied.

From the linkage vessel condition, we get:
\begin{equation} \label{link}
\gamma_{jk} = \gamma_{*jk} - i \sigma_j \Phi \Phi^* \sigma_k + i \sigma_k \Phi \Phi^* \sigma_j.
\end{equation}

To prove that the fourth compatibility condition of \eqref{eq:output_VR} holds, we use the linkage condition to obtain:
\begin{equation*}
\gamma_{*jk} = \gamma_{jk} + i \sigma_j \Phi \Phi^* \sigma_k - i \sigma_k \Phi \Phi^* \sigma_j.
\end{equation*}
Then equation \eqref{eq:additional_condition} yields:
\begin{equation} \label{partial-four}
\gamma_{*jk} = \sigma_j \sigma_1^{-1} \gamma_{1k} - \sigma_k \sigma_1^{-1} \gamma_{1j}  i \sigma_j \Phi \Phi^* \sigma_k - i \sigma_k \Phi \Phi^* \sigma_j.
\end{equation}
Now using \eqref{link}, we get:
\begin{align} \label{mega-link}
\begin{split}
& \sigma_j \sigma_1^{-1} \gamma_{1k} = \sigma_j \sigma_1^{-1} \gamma_{*1k} - i \sigma_j \Phi \Phi^* \sigma_k + i \sigma_j \sigma_1^{-1} \sigma_k \Phi \Phi^* \sigma_1 \\
& \gamma_{1k} \sigma_1^{-1} \sigma_j = \sigma_j \sigma_1^{-1} \gamma_{*1k} - i \sigma_1 \Phi \Phi^* \sigma_k \sigma_1^{-1} \sigma_j + i \sigma_k \Phi \Phi^* \sigma_j\\
\end{split}
\end{align}
Plugging in the equations of \eqref{mega-link} into \eqref{partial-four}, and using the first condition of \eqref{eq:commutation_conditions}, we obtain the fourth equation of \eqref{eq:output_VR}.

Similarly we consider the third equation of \eqref{eq:commutation_conditions}. We use \eqref{mega-link} and obtain immediately the third condition of \eqref{eq:output_VR}.

Using the vessel conditions one obtains the following equation:
\begin{align} \label{mux}
\begin{split}
& \sigma_j \Phi \Phi^* \sigma_k \Phi \Phi^* \sigma_l = i \sigma_j \Phi A_k^* \Phi^* \sigma_k - i \sigma_j \Phi A_k \Phi^* \sigma_l \\
& i \gamma_{*jk} \Phi \Phi^* \sigma_l = i \sigma_k \Phi A A_j \Phi^* \sigma_l - i \sigma_j \Phi A_k \Phi^* \sigma_l \\
& i \sigma_j \Phi \Phi^* \gamma_{*kl} = i \sigma_j \Phi A_k^* \Phi^* \sigma_l - i \sigma_j \Phi A_l^* \Phi^* \sigma_k.
\end{split}
\end{align}
Using \eqref{mega-link} on the second condition of \eqref{eq:commutation_conditions}. we get:
\begin{align}
\begin{split}
& 0 = \gamma_{*1j} \sigma_1^{-1} \gamma_{*1k} - \gamma_{*1k} \sigma_1^{-1} \gamma_{*1j} - i \gamma_{*1j} \Phi \Phi^* \sigma_k + \\
& + i \gamma_{*1k} \Phi \Phi^* \sigma_j +i (\gamma_{*1j} \sigma_1^{-1} \sigma_k - \gamma_{*1k} \sigma_1^{-1} \sigma_j) \Phi \Phi^* \sigma_1 + \\
& + i \sigma_j \Phi \Phi^* \gamma_{*1k} - i \sigma_3 \Phi \Phi^* \gamma_{*1j} + i \sigma_1 \Phi \Phi^* (\sigma_k \sigma_1^{-1} \gamma_{*1j} - \sigma_j \sigma_1^{-1} \gamma_{*1k}) - \\
& - \sigma_1 \Phi \Phi^* \sigma_j \Phi \Phi^* \sigma_k + \sigma_j \Phi \Phi^* \sigma_1 \Phi \Phi^* \sigma_k - \sigma_j \Phi \Phi^* \sigma_j \Phi \Phi^* \sigma_1 + \\
& + \sigma_1 \Phi \Phi^* \sigma_j \Phi \Phi^* \sigma_k - \sigma_k \Phi \Phi^* \sigma_1 \Phi \Phi^* \sigma_j + \sigma_k \Phi \Phi^* \sigma_j \Phi \Phi^* \sigma_1
\end{split}
\end{align}
Now using the fourth equation of \eqref{eq:output_VR} and \eqref{mux}, we see that all the terms cancel, but for the first two. Thus we have obtained the second equation of \eqref{eq:output_VR} and the proposition is proved.
\end{proof}

Now that we know that the $VR$ conditions fit naturally into the framework of vessels, we ask a question about invariance under coordinate changes. Namely, we assume that $\sigma_2$ is invertible as well and we can write a system of $VR$ conditions for $\sigma_2$, for $j,k =1,3,4,\ldots,d$:
\begin{align} \label{eq:sigma2_VR}
\begin{split}
& \bullet [\sigma_2^{-1} \sigma_j, \sigma_2^{-1} \sigma_k] = 0, \\
& \bullet [\sigma_2^{-1} \gamma_{2j}, \sigma_2^{-1} \gamma_{2k}] = 0, \\
& \bullet [\sigma_2^{-1} \sigma_k, \sigma_2^{-1} \gamma_{2j}] = [\sigma_2^{-1} \sigma_j, \sigma_2^{-1} \gamma_{2k}],\\
& \bullet \gamma_{jk} = \sigma_j \sigma_2^{-1} \gamma_{2k} - \sigma_k \sigma_2^{-1} \gamma_{2j}.
\end{split}
\end{align}
We will call this system of conditions $VR$ conditions in the direction of $e_2$ and we refer to the original $VR$ conditions as $VR$ conditions in the direction of $e_1$.

\begin{prop} \label{prop:coord_change}
Given a vessel $\fV$ such that both $\sigma_1$ and $\sigma_2$ are invertible, then the $VR$ conditions in the direction of $e_1$ are satisfied if and only if the $VR$ conditions in the direction of $e_2$ are satisfied.
\end{prop}
\begin{proof}
Let us assume that the $VR$ conditions in the direction of $e_2$ hold. Then for every $k = 3,\ldots,d$ we have that:
\[
\sigma_1 \sigma_2^{-1} \sigma_k = \sigma_k \sigma_2^{-1} \sigma_1.
\]
Premultiplying by $\sigma_1^{-1}$ we get:
\[
\sigma_2^{-1} \sigma_k = \sigma_1^{-1} \sigma_k \sigma_2^{-1} \sigma_1.
\]
Hence for every $j,k = 3,\ldots,d$ we get:
\[
\sigma_j \sigma_2^{-1} \sigma_k = \sigma_j \sigma_1^{-1} \sigma_k \sigma_2^{-1} \sigma_1.
\]
Now using the first $VR$ condition for $\sigma_2$ we get:
\[
\sigma_j \sigma_1^{-1} \sigma_k \sigma_2^{-1} \sigma_1 = \sigma_k \sigma_1^{-1} \sigma_j \sigma_2^{-1} \sigma_1.
\]
This gives us the first $VR$ condition in the direction of $e_1$, namely the first equation of \eqref{eq:commutation_conditions} for $j,k = 3,\ldots,d$. We only need to check for $j=2$ and $k=3,\ldots,d$:
\begin{multline*}
\sigma_2 \sigma_1^{-1} \sigma_k - \sigma_k \sigma_1^{-1} \sigma_2 = \sigma_2 \sigma_1^{-1} \left( \sigma_k - \sigma_1 \sigma_2^{-1} \sigma_k \sigma_1^{-1} \sigma_2 \right) = \\ \sigma_2 \sigma_1^{-1} \left(\sigma_k - \sigma_k \sigma_2^{-1} \sigma_1 \sigma_1^{-1} \sigma_2 \right) = 0.
\end{multline*}
We get the second equation of \eqref{eq:commutation_conditions} in exactly the same way. Combining the two we get easily the third equation of \eqref{eq:commutation_conditions}.

Now we use the fourth equation of \eqref{eq:sigma2_VR} to get that for every $k=3,\ldots,d$:
\[
\gamma_{1k} = \sigma_1 \sigma_2^{-1} \gamma_{2k} - \sigma_k \sigma_2^{-1} \gamma_{21}.
\]
Therefore for $j,k = 2,\ldots,d$:
\begin{multline*}
\sigma_j \sigma_1^{-1} \gamma_{1k} - \sigma_k \sigma_1^{-1} \gamma_{1j} = \sigma_j \sigma_2^{-1} \gamma_{2k} - \sigma_j \sigma_1^{-1} \sigma_k \sigma_2^{-1} \gamma_{21} \\ - \sigma_k \sigma_2^{-1} \gamma_{2j} + \sigma_k \sigma_1^{-1} \sigma_j \sigma_2^{-1} \gamma_{21} = \gamma_{jk}. 
\end{multline*}
Here we have used the first equation of \eqref{eq:commutation_conditions} and the fourth equation of \eqref{eq:sigma2_VR} for $j,k = 3,\ldots,d$. For $j = 2$ and $k = 3,\ldots,d$ we get:
\begin{multline*}
\sigma_2 \sigma_1^{-1} \gamma_{1k} - \sigma_k \sigma_1^{-1} \gamma_{12} = \\ \gamma_{2k} - \sigma_2 \sigma_1^{-1} \sigma_k \sigma_2^{-1} \gamma_{21} - \sigma_k \sigma_1^{-1} \gamma_{12} = \gamma_{2k}.
\end{multline*}
Hence we proved that \eqref{eq:additional_condition} holds.
\end{proof}

Given a vessel $\fV$ we can consider as in \cite{SV14a} the linear maps $\rho \colon \R^d \to B(\cH)$, $\sigma \colon \R^d \to B(\cH)$ and $\gamma \colon \wedge^2 \R^d \to B(\cH)$ given by:
\[
\rho(e_j) = A_j, \, \sigma(e_j) = \sigma_j, \, \gamma(e_j \wedge e_k) = \gamma_{jk}.
\]
Then for every $T \in \GL_d(\R)$ we can define:
\[
A_j^T = \rho(T e_j), \, \sigma_j^T = \sigma(T e_j), \, \gamma_{jk}^T = \gamma (\wedge^2(T) e_j \wedge e_k).
\]
Thus we get a vessel $\fV^T$, we call this the coordinate change corresponding to $T$. We will say that $\fV$ satisfies the $VR$ conditions in the direction of $T e_1$ if $\fV^T$ satisfies the $VR$ conditions in the direction of $e_1$, generalizing \eqref{eq:sigma2_VR}.
\begin{cor} \label{cor:generic_change}
Let us assume that there exist $\xi, \eta \in \R^d$, such that both $\xi \sigma = \sum_{j=1}^d \xi_j \sigma_j$ and $\eta \sigma$ are invertible and $\fV$ satisfies the $VR$ conditions in the direction of $\xi$ then it satsfies the $VR$ conditions in the direction of $\eta$. In particular, there exists an open set $U \subset \GL_d(\R)$, such that for every $T \in U$, $\fV$ satisfies the $VR$ conditions in the direction of $T \xi$.
\end{cor}
\begin{proof}
We can take $T \in \GL_d(\R)$, such that $T \xi = e_1$ and $T \eta = e_2$ and apply Proposition \ref{prop:coord_change}. To obtain the second part of the statement we note that since the invertible matrices are an open set, for $\eta \in \RR^d$, such that $\|\xi \sigma - \eta \sigma\|$ is small enough, we have that $\eta \sigma$ is invertible.
\end{proof}

This corollary allows us to treat $VR$ conditions without mentioning the direction. For definiteness we will assume for the rest of this section that the $VR$ conditions in direction $e_1$ are satisfied.

\begin{rem}
In case $\dim \cE < \infty$ one notes that the set $U$ from the above Corollary is in fact Zariski open and dense.
\end{rem}

The $VR$ conditions are slightly redundant as the following proposition shows:
\begin{prop} \label{prop:fourth_implies_third}
Assume that $\fV$ is a vessel that satisfies \eqref{eq:additional_condition}, then it satisfies the third condition of \eqref{eq:commutation_conditions} automatically.
\end{prop}
\begin{proof}
Since $\gamma_{jk}$ and $\sigma_j$ are selfadjoint we get by taking the adjoint of \eqref{eq:additional_condition} that:
\[
\gamma_{jk} = \gamma_{1k} \sigma_1^{-1} \sigma_j - \gamma_{1j} \sigma_1^{-1} \sigma_k.
\]
Now subtract it from \eqref{eq:additional_condition} to get the third equation of \eqref{eq:commutation_conditions}.
\end{proof}

The following is a strong converse to Proposition \ref{prop:fourth_implies_third} and provides a way to construct $VR$ vessels from partial data:

\begin{prop}
Assume that we are given a $d$-tuple of commuting non-selfadjoint operators $A_1,\ldots,A_d$ on $\cH$, an operator $\Phi \colon \cH \to \cE$ and a collection of selfadjoint operators $\sigma_1,\ldots,\sigma_d$ and $\gamma_{12},\ldots,\gamma_{1d}$ on $\cE$, such that $\sigma_1$ is invertible, the commutativity conditions \eqref{eq:commutation_conditions} hold and the relevant vessel conditions hold, namely for every $j = 1,\ldots,d$:
\[
A_j - A_j^* = i \Phi^* \sigma_j \Phi.
\]
\[
\gamma_{1j} \Phi = \sigma_1 \Phi A_j^* - \sigma_j \Phi A_1^*.
\]
Then there exists a $VR$ vessel $\fV$ with the above data.
\end{prop}
\begin{proof}
We define $\gamma_{jk}$ using \eqref{eq:additional_condition}, namely:
\[
\gamma_{jk} = \sigma_j \sigma_1^{-1} \gamma_{1k} - \sigma_k \sigma_1^{-1} \gamma_{1j}.
\]
Then using the same computation as in the preceding Proposition we see that $\gamma_{jk}$ is selfadjoint. Now to see that it satisfies the input vessel condition we check:
\begin{multline*}
\gamma_{jk} \Phi = \sigma_j \sigma_1^{-1} \gamma_{1k} \Phi - \sigma_k \sigma_1^{-1} \gamma_{1j} \Phi = \\ \sigma_j \sigma_1^{-1} \left( \sigma_1 \Phi A_k^* - \sigma_k \Phi A_1^*\right) - \sigma_k \sigma_1^{-1} \left( \sigma_1 \Phi A_j^* - \sigma_j \Phi A_1^*\right) = \sigma_j \Phi A_k^* - \sigma_k \Phi A_j^*.
\end{multline*}
The last equality follows from the first condition of \eqref{eq:commutation_conditions}. Now we define $\gamma_{*jk}$ using the linkage condition to get a vessel $\fV$. It is obvious that this vessel satisfies the $VR$ conditions.
\end{proof}

Recall that from Definition \ref{dfn:strict_vessel} a strict vessel is a vessel, such that $\Phi$ is surjective and $\cap_{j=1}^d \ker \sigma_j = 0$. Since we assume that $\sigma_1$ is invertible, the second condition holds automatically. The strict vessels are slightly easier to work with as the following claim shows:

\begin{prop} \label{prop:strict}
Assume that $\fV$ is a strict vessel that satisfies the first condition of \eqref{eq:commutation_conditions}, namely $\sigma_j \sigma_1^{-1} \sigma_k = \sigma_k \sigma_1^{-1} \sigma_j$. Then $\fV$ satisfies the $VR$ conditions.
\end{prop}
\begin{proof}
Since the vessel is strict $\Phi$ is surjective and $\Phi^*$ is injective. We can assume without loss of generality that $\Phi \Phi^* = I_{\cE}$. Hence from the vessel conditions we get that:
\begin{multline*}
\sigma_j \sigma_1^{-1} \gamma_{1k} \Phi - \sigma_k \sigma_1^{-1} \gamma_{1j} \Phi - \gamma_{jk} \Phi = \\
 \sigma_j \sigma_1^{-1} \left( \sigma_1 \Phi A_k^* - \sigma_k \Phi A_1^*\right) - \sigma_k \sigma_1^{-1} \left( \sigma_1 \Phi A_j^* - \sigma_j \Phi A_1^*\right) - \sigma_j \Phi A_k^* + \sigma_k \Phi A_j^* = \\
  \left( \sigma_j \sigma_1^{-1} \sigma_k - \sigma_k \sigma_1^{-1} \sigma_j \right) \Phi A_1^*.
\end{multline*}
Now postmultiplying by $\Phi^*$ we obtain:
\[
\sigma_j \sigma_1^{-1} \gamma_{1k} - \sigma_k \sigma_1^{-1} \gamma_{1j} - \gamma_{jk} = \left( \sigma_j \sigma_1^{-1} \sigma_k - \sigma_k \sigma_1^{-1} \sigma_j \right) \Phi A_1^* \Phi^*.
\]
In particular \eqref{eq:additional_condition} follows from the first condition of \eqref{eq:commutation_conditions}. Next we compute:
\begin{multline*}
\Phi^* \gamma_{1j} \sigma_1^{-1} \gamma_{1k} \Phi = (A_j \Phi^* \sigma_1 - A_1 \Phi^* \sigma_j) \sigma_1^{-1} (\sigma_1 \Phi A_k^* - \sigma_k \Phi A_1^*) = \\ A_j \Phi^* \sigma_1 \Phi A_k^* - A_j \Phi^* \sigma_k \Phi A_1^*  - A_1 \Phi^* \sigma_j \Phi A_k^* + A_1 \Phi^* \sigma_j \sigma_1^{-1} \sigma_k \Phi A_1^* = \\ \frac{1}{i} (A_j A_1 A_k^* - A_j A_1^* A_k^*) - \frac{1}{i} (A_j A_k A_1^* - A_j A_k^* A_1^*) - \frac{1}{i} (A_1 A_j A_k^* - A_1 A_j^* A_k^*) + \\ A_1 \Phi^* \sigma_j \sigma_1^{-1} \sigma_k \Phi A_1^* = \frac{1}{i} (A_1 A_j^* A_k^* - A_j A_k A_1^*) + A_1 \Phi^* \sigma_j \sigma_1^{-1} \sigma_k \Phi A_1^*.
\end{multline*} 
Similarly we get:
\[
\Phi^* \gamma_{1k} \sigma_1^{-1} \gamma_{1j} \Phi = \frac{1}{i} (A_1 A_k^* A_j^* - A_k A_j A_1^*) + A_1 \Phi^* \sigma_k \sigma_1^{-1} \sigma_j \Phi A_1^*
\]
Now since the $A_j$ commute we obtain after premultiplying by $\Phi$ and postmultiplying by $\Phi^*$ that:
\[
\gamma_{1j} \sigma_1^{-1} \gamma_{1k} - \gamma_{1k} \sigma_1^{-1} \gamma_{1j} = \Phi A_1 \Phi^* \left( \sigma_j \sigma_1^{-1} \sigma_k - \sigma_k \sigma_1^{-1} \sigma_j \right) \Phi A_1^* \Phi^*.
\]
In particular the second condition of \eqref{eq:commutation_conditions} follows from the first. Now using Proposition \ref{prop:fourth_implies_third} we get the result.
\end{proof}

\begin{cor} \label{cor:doubly_commuting}
Assume that $A_1,\ldots,A_d$ are doubly commuting (i.e., $[A_j,A_k^*] = 0$ for every $j \neq k$) and that there exists $\xi \in \R^d$, such that $\sum_{j=1}^d \xi_j (A_j - A_j^*)$ is invertible when restricted to the non-Hermitian subspace, then the strict vessel they embed into satisfies the $VR$ conditions.
\end{cor}
\begin{proof}
From the formulae \eqref{eq:strict_vessel} it follows that the $\sigma_j$ commute and our assumption implies that there exists $\xi \in \R^d$, such that $\xi \sigma$ is invertible. Now apply Proposition \ref{prop:strict} to get the result. 
\end{proof}

\begin{rem}
Let $A_1,\ldots,A_d$ be a $d$-tuple of commuting operators, then Proposition \ref{prop:strict} and Corollary \ref{cor:doubly_commuting} imply that the assumption that the strict vessel satisfies the $VR$ conditions is a generalization of the doubly-commuting property.
\end{rem}

\section{Solution in the Hyperbolic Case} \label{sec:cont_solve}

In this section we study the hyperbolic case and thus from now on we assume that there exists an $\epsilon > 0$, such that $\sigma_1 > \epsilon I$. Let us write again $\alpha_j = \sigma_1^{-1} \sigma_j$ and $\beta_j = \sigma_1^{-1} \gamma_{1j}$ and set $\alpha_1 = I_{\cE}$ and $\beta_1 = 0$. In this case $\alpha_j$ and $\beta_j$ are selfadjoint with respect to the $\sigma_1$-inner product on $\cE$. Without loss of generality we may assume that $\sigma_1 = I_{\cE}$, since otherwise we can simply replace the inner product on $\cE$ by the $\sigma_1$-inner product and then the $\sigma_j$ will be replaced by $\alpha_j$ and $\gamma_{1j}$ will be replaced by $\beta_j$. For $x \in \R^d$ let us write $\alpha(x) = \sum_{j=1}^d x_j \alpha_j$ and similarly $\beta(x) = \sum_{j=1}^d x_j \beta_j$. Since $\alpha_1 = I$ there exists $\epsilon_j > 0$ small enough such that $\alpha_1 + \epsilon_j \alpha_j > \delta I$ for every $j = 2,\ldots,d$ and some $\delta > 0$. Hence by changing coordinates we may assume that $\alpha_j > \delta I$ for every $j = 2,\ldots,d$. The following definition describes the future cone of our system:
\begin{dfn} \label{dfn:positive_cone}
Let $\fV$ be a vessel, define the following set in $\R^d$:
\[
\Pos(\fV) = \{ \xi \in \R^d \mid \exists \epsilon > 0: \, \xi \sigma = \sum_{j=1}^d \xi_j \sigma_j > \epsilon I_{\cE} \}.
\]
Note that $\Pos(\fV)$ is either empty or an open convex cone in $\R^d$. 
\end{dfn}

Recall that by a theorem of Grothendieck a function with values in $\cE$ is smooth if and only if it is weakly smooth (cf. \cite[Sec.\ 3.8]{Gro73} or \cite{Gro53}). Denote by $\schwartz$ the Schwarz space of $\cE$-valued rapidly decreasing smooth functions on $\R$. Namely, $\schwartz$ is the space of smooth $\cE$-valued functions, such that for every two polynomials $P$ and $Q$ we have that $\|P(t) Q(\pd{}{t}) f\|$ is bounded on $\R$. By \cite[Thm.\ 44.1]{Tre67} and \cite[Ex.\ 44.6]{Tre67} we have that $\schwartz \cong \cS(\R) \widehat{\otimes} \cE$, where the choice of the completed tensor product does not matter since $\cS(\R)$ is a nuclear Frechet space (cf. \cite[Ch.\ 51]{Tre67}). We also consider the space of tempered $\cE$-distributions on $\R$, namely the topological dual of $\schwartz$, that we will denote by $\tempdist$. Since our goal is to discuss operators on Hilbert spaces we will use an anti-linear pairing between tempered distributions and Schwarz functions. We note that we can endow $\tempdist$ with the strong topology of uniform convergence on bounded subsets and that by \cite[Prop.\ 50.7]{Tre67} we have $\tempdist \cong \cS^{\prime}(\R) \widehat{\otimes} \cE$ (the identification is again anti-linear). We can also endow $\tempdist$ with the weak topology of pointwise convergence and those topologies will coincide if and only if $\dim \cE < \infty$. Similarly we define $\schwartzd$ and $\tempdistd$. We will also use the space $\Ltwoe$ that is the space of all weakly measurable functions $f \colon \R \to \cE$ (this is equivalent by Pettis' theorem to strongly measurable since $\cE$ is separable), such that:
\[
\|f\|_{L^2}^2 = \int_{-\infty}^{\infty} \|f(t)\|^2 dt < \infty.
\]
Similarly, if $\alpha$ is an invertible positive-definite operator on $\cE$ we will define the space $L^2(\R,\cE,\alpha)$ as the set of all weakly measurable functions $f \colon \R \to \cE$, such that:
\[
\int_{-\infty}^{\infty} \langle \alpha f(t), f(t) \rangle dt < \infty.
\]
Then in particular we have that $\Ltwoe \cong L^2(\R) \widehat{\otimes}_H \cE$, where $\widehat{\otimes}_H$ is the tensor product of Hilbert spaces. We have a continuous embedding $\schwartz \hookrightarrow \Ltwoe$ and its image is dense. We can define the Fourier transform by considering the continuous linear map $\cF \otimes I_{\cE}$ on $\cS(\R) \widehat{\otimes} \cE$. This is equivalent to:
\[
\cF(f)(t) = \frac{1}{\sqrt{2 \pi}} \int_{-\infty}^{\infty} e^{- i s t} f(s) ds.
\]
Here the integral is considered as a Gel'fand-Pettis integral and by the same consideration as in the classical Plancherel theorem it extends to an isometric automorphism of $\Ltwoe$.

Let us assume now that the input signal is a Schwarz $\cE$-valued function on $\R^d$ and assume that the initial condition is $u = f$ on the $t_1$-axis. Let us then apply the Fourier transform along the $t_1$-axis to $u$ and write $\cF_1(u) = \widehat{u}$. Then we get a system of equations ($j = 2,\ldots,d$):
\[
\pd{\widehat{u}}{t_j} = i \alpha_j \tau_1 \widehat{u} + i \beta_j \widehat{u}.
\]
Here $\tau_1$ is the variable in the frequency domain. The initial condition is $\widehat{u} = \widehat{f}$ on the $\tau_1$-axis. Each of these equations has a solution of the form: 
\[
\varphi_j (\tau_1,t_2,\ldots,t_d) = e^{i t_j \left( \alpha_j \tau_1 + \beta_j \right)} C_j(\tau_1,t_2,\ldots, t_{j-1},t_{j+1}, \ldots, t_d),
\]
where $C_j$ is an $\cE$-valued function. One then can proceed plugging one solution into the other equations and then using the initial condition. Since the equations are compatible and the pencils in the exponent commute by \eqref{eq:commutation_conditions}, we will get a solution:
\[
\widehat{u} = e^{\sum_{j=2}^d i t_j \left( \alpha_j \tau_1 + \beta_j\right)} \widehat{f}(\tau_1).
\]
Hence a solution to the system of input compatibility equations is:
\begin{equation} \label{eq:fourier_solution}
u(t_1,\ldots,t_d) = \frac{1}{\sqrt{2 \pi}} \int_{-\infty}^{\infty} e^{\sum_{j=1}^d i t_j \left( \alpha_j \tau_1 + \beta_j\right)} \widehat{f}(\tau_1) d\tau_1.
\end{equation}
This computation makes sense for Schwarz functions. We will show below (see Corollary \ref{lem:sch_maps_to_smooth}) that for every $f \in \schwartz$ the above formula defines a smooth $\cE$-valued function on $\R^d$ (not necessarily Schwarz) that is a solution of our system. Next we would like to extend it to a wider class of functions on $\R$.

Note that $\R^d$ acts on $\schwartz$ by $(t e_j \cdot \varphi)(s) = e^{i t (s \alpha_j + \beta_j)} \varphi(s)$. This function is clearly Schwarz since both $\alpha_j$ and $\beta_j$ are selfadjoint and hence the exponent is a unitary operator on $\cE$. We can conjugate this action by the Fourier transform, namely we get a representation of $\R^d$:
\begin{equation} \label{eq:pi}
\pi(t e_j) \varphi = \cF^{-1}(e^{i t (s \alpha_j + \beta_j)} \cF(\varphi)).
\end{equation}
Note that this representation is smooth by virtue of a theorem of Bruhat (\cite[Prop.\ 4.4.1.7]{War72} and the following fact:
\[
\pi(\pd{}{t_1}) \varphi = \varphi^{\prime}.
\]
\[
\pi(\pd{}{t_j}) \varphi = \alpha_j \varphi^{\prime} + i \beta_j \varphi.
\]
By \cite[Prop.\ 4.4.1.9]{War72} we have that the contragredient representation of $\R^d$ on $\tempdist$ is also smooth, since $\tempdist$ is complete. Hence we get for every $f \in \tempdist$ a smooth function $L_f \colon \R^{d-1} \to \tempdist$ that is evaluated as:
\[
\langle L_f(t_2,\ldots,t_d)(\cdot),\varphi \rangle = \langle f , \cF(e^{- i \sum_{j=2}^d t_j (s \alpha_j + \beta_j)} \cF^{-1}(\varphi)) \rangle.
\]
Note that if $u \in \cS(\R^d,\cE)$, that solves the system of input compatibility conditions and $f$ is its restriction to the $t_1$-axis, then \eqref{eq:fourier_solution} implies that $u(\cdot,t_2,\ldots,t_d) = L_f(t_2,\ldots,t_d)(\cdot)$. Note also, that in general the function $L_f$ solves the input compatibility conditions in the following sense:
\[
\pi(\pd{}{t_j})L_f = \alpha_j L_{f^{\prime}} + i \beta_j L_f = (\alpha_j \pi(\pd{}{t_1}) + i \beta_j) L_f.
\]

\begin{rem}
Recall that equation \eqref{eq:additional_condition} implies that the remaining equations are satisfied, if those involving $\sigma_1$ are satisfied. 
\end{rem}

\begin{prop} \label{prop:dist_on_R^d}
The function $L_f$ defines a tempered $\cE$-valued distribution $\fu_f$ on $\R^d$ as follows: 
\begin{equation} \label{eq:u_as_dist}
\langle \fu_f, \psi \rangle = \int_{\R^{d-1}} \langle L_f(t_2,\ldots,t_d)(\cdot), \psi(\cdot, t_2,\ldots,t_d) \rangle dt_2 \cdots dt_d.
\end{equation}
Here $\psi \in \schwartzd$ is a Schwarz function.
\end{prop}

\begin{proof}
We note that since $f \in \tempdist$ there exist polynomials $P(t)$ and $Q(t)$ and a constant $C > 0$, such that for every choice of $t_2,\ldots,t_d$ and $\psi$, we have:
\begin{multline*}
|\langle L_f(t_2,\ldots,t_d)(\cdot), \psi(\cdot, t_2,\ldots,t_d) \rangle| = \\ | \langle f , \cF(e^{-i \sum_{j=2}^d t_j (s \alpha_j + \beta_j)} \cF^{-1}(\psi(\cdot,t_2,\ldots,t_d))) \rangle |\leq \\ C \sup_{t_1 \in \R} \|P(t_1) Q(\pd{}{t_1}) \cF(e^{-i \sum_{j=2}^d t_j (s \alpha_j + \beta_j)} \cF^{-1}(\psi(t_1,t_2,\ldots,t_d)))\| = \\
C \sup_{t_1 \in \R} \| \cF( P(i \pd{}{s}) Q(-i s)e^{-i \sum_{j=2}^d t_j (s \alpha_j + \beta_j)} \cF^{-1}(\psi(t_1,t_2,\ldots,t_d)))\| = \\
C \sup_{t_1 \in \R} \| \cF( P(i \pd{}{s}) e^{-i \sum_{j=2}^d t_j (s \alpha_j + \beta_j)} \cF^{-1}(Q(\pd{}{t_1}) \psi(t_1,t_2,\ldots,t_d)))\| \leq \\
C C^{\prime} \sup_{s \in \R} \| P(i \pd{}{s}) e^{-i \sum_{j=2}^d t_j (s \alpha_j + \beta_j)} \cF^{-1}((1 + \pd{}{t_1^2})Q(\pd{}{t_1}) \psi(t_1,t_2,\ldots,t_d)))\|.
\end{multline*}
The last inequality is due to the fact that for every function $\varphi \in \schwartz$ we have that:
\begin{equation} \label{eq:fourier_estimate}
\sup_{s\in\R}\| \cF(\varphi)(s)\| \leq C^{\prime} \sup_{s \in \R} \| (1 + s^2) \varphi(s)\|.
\end{equation}
For the derivative of the exponent applied to a Schwarz function we have the following bound (here we write $\eta = \cF^{-1}((1 + \pd{}{t_1^2})Q(\pd{}{t_1}) \psi(t_1,t_2,\ldots,t_d))$):
\[
\|\pd{}{s} \left(e^{-i \sum_{j=2}^d t_j (s \alpha_j + \beta_j)} \eta(s)\right)\| \leq \| \pd{}{s} \left(e^{-i \sum_{j=2}^d t_j (s \alpha_j + \beta_j)} \right) \eta(s) \| + \|\pd{\eta}{s} \| \]
We use here the fact that the exponent is a unitary operator for every choice of real $s$. Now we have the following well known equality:
\[
\pd{}{s} \left(e^{-i \sum_{j=2}^d t_j (s \alpha_j + \beta_j)} \right) = - \int_0^1 e^{-i w \sum_{j=2}^d t_j (s \alpha_j + \beta_j)} \left( \sum_{j=2}^d t_j \alpha_j \right) e^{-(1-w) i\sum_{j=2}^d t_j (s \alpha_j + \beta_j)} dw. 
\]
Thus we obtain the inequality:
\[
\|\pd{}{s} \left(e^{- i \sum_{j=2}^d t_j (s \alpha_j + \beta_j)} \eta(s)\right)\| \leq \sum_{j=2}^d |t_j| \| \alpha_j\| \| \eta(s)\| + \| \pd{\eta}{s} \|.
\]
Notice that in the expression $\pd{\eta}{s}$ we can push $\pd{}{s}$ into $\cF^{-1}$ replacing it by $-i t_1$. Using estimate \eqref{eq:fourier_estimate} we can then get rid of the Fourier transform. Applying these considerations to every monomial in $P$ we eventually get that there exist a constant $\tilde{C} > 0$ and polynomials $\tilde{P}, \tilde{Q} \in \C[t_1,\ldots,t_d]$, such that:
\[
|\langle L_f(t_2,\ldots,t_d)(\cdot), \psi(\cdot, t_2,\ldots,t_d) \rangle| \leq \tilde{C} \sup_{t_1 \in \R} \| \tilde{P}(t) \tilde{Q}(\pd{}{t}) \psi(t)\|.
\]
As for the integral we write:
\begin{multline*}
\left|\int_{\R^{d-1}} \langle L_f(t_2,\ldots,t_d)(\cdot), \psi(\cdot, t_2,\ldots,t_d) \rangle dt_2\cdots dt_d \right| \leq \\ C_0 \sup_{t = (t_1,\ldots,t_d) \in \R^d)} \|(1 + t_2^2 + \ldots + t_d^2)^d \tilde{P}(t) \tilde{Q}(\pd{}{t}) \psi(t)\|.
\end{multline*}
Here $C_0 = \tilde{C} \int_{\R^{d-1}} (1 + t_2^2 + \ldots + t_d^2)^{-d} dt_2\cdots t_d$.
\end{proof}

\begin{prop} \label{prop:weak_solution}
The tempered distribution $\fu_f$ defined in \eqref{eq:u_as_dist} is a weak solution for the system of input compatibility conditions.
\end{prop}
\begin{proof}
We want to show that for every $f \in \tempdist$ and every $\psi \in \schwartz$, we have:
\[
\langle \pd{\fu_f}{t_k}, \psi \rangle = \langle \alpha_k \pd{\fu_f}{t_1} + i \beta_k \fu_f, \psi \rangle.
\]
Since $\alpha_k$ and $\beta_k$ are selfadjoint, by taking the adjoint we see that the desired equality becomes:
\[
\langle \fu_f, \pd{\psi}{t_k} \rangle = \langle u, \alpha_k \pd{\psi}{t_1} + i \beta_k \psi \rangle.
\]

For $\psi \in \schwartzd$, we want to compute $\langle \fu_f, \pd{}{t_k} \psi \rangle$, for some $k = 2, \ldots,d$. For fixed $t_2,\ldots,t_d \in \R^d$ we have:
\[
\langle L_f(t_2,\ldots,t_d)(\cdot),\pd{}{t_k} \psi \rangle = \langle f , \cF( e^{-i \sum_{j=2}^d t_j (s\alpha_j + \beta_j)} \cF^{-1} (\pd{}{t_k} \psi)) \rangle.
\]
Now we compute:
\begin{multline*}
\pd{}{t_k} \left( \cF( e^{-i \sum_{j=2}^d t_j (s\alpha_j + \beta_j)} \cF^{-1} (\psi) \right) = \\ \cF\left( \pd{}{t_k} \left( e^{-i \sum_{j=2}^d t_j (s\alpha_j + \beta_j)} \right) \cF^{-1} (\psi) \right)  + \cF( e^{-i \sum_{j=2}^d t_j (s\alpha_j + \beta_j)} \cF^{-1} (\pd{\psi}{t_k})) = \\- \cF( e^{-i \sum_{j=2}^d t_j (s\alpha_j + \beta_j)} (i s \alpha_k + i \beta_k) \cF^{-1}(\psi)) + \cF( e^{-i \sum_{j=2}^d t_j (s\alpha_j + \beta_j)} \cF^{-1} (\pd{\psi}{t_k})) = \\ -\cF(e^{-i \sum_{j=2}^d t_j (s\alpha_j + \beta_j)} \cF^{-1}(\alpha_k \pd{\psi}{t_1} + i \beta_k \psi))  + \cF( e^{-i \sum_{j=2}^d t_j (s\alpha_j + \beta_j)} \cF^{-1} (\pd{\psi}{t_k})) .
\end{multline*}
Now we apply $f$ and integrate on $\R^{d-1}$ to get:
\begin{multline*}
\int_{\R^{d-1}} \langle f, \pd{}{t_k} \left( \cF( e^{-i \sum_{j=2}^d t_j (s\alpha_j + \beta_j)} \cF^{-1} (\psi(\cdot, t_2,\ldots,t_d) \right)\rangle = \\- \langle 
fu_f, \alpha_k \pd{\psi}{t_1} + i \beta_k \psi \rangle + \langle \fu_f, \pd{\psi}{t_k} \rangle.
\end{multline*}
It remains to note that the left hand side is zero since $\psi$ is a Schwarz function on $\R^d$ and $f$ is a distribution on $\R$.

\end{proof}

\begin{lem} \label{lem:sch_maps_to_smooth}
If $f \in \schwartz$ then $\fu_f$ is a smooth function on $\R^d$ that solves the system of input compatibility equations. Furthermore, $u_f$ is given by the formula \eqref{eq:fourier_solution}.
\end{lem}
\begin{proof}
Let $f \in \schwartz$ and note that by definition:
\[
L_f(t_2,\ldots,t_d)(\cdot) = (\pi(\cdot,t_2,\ldots,t_d) f)(0).
\]
Here $\pi$ is the representation defined in \eqref{eq:pi}. Therefore, the associated $u$ is the smooth function given by \eqref{eq:fourier_solution} and since it is a weak solution it is a solution. 
\end{proof}

\begin{lem} \label{lem:extend_to_L2}
The representation $\pi$ (defined by \eqref{eq:pi}) of $\R^d$ on $\schwartz$ extends to a unitary representation of $\R^d$ on $\Ltwoe$.
\end{lem}
\begin{proof}
Since both $\alpha_j$ and $\beta_j$ are selfadjoint for every $j = 1,\ldots,d$ and $t_j$ and $s$ are real, the multiplication by $e^{i \sum_{j=1}^d t_j ( s \alpha_j + \beta_j)}$ is a unitary operator.
\end{proof}

To better understand the solutions we will study their behavior on lines. Consider the formula \eqref{eq:fourier_solution}. If we fix a line $\tau x + y$ in $\R^d$, then we define for every $f \in \schwartz$ the following operator:
\begin{equation} \label{eq:fourier_on_line}
(\Lambda(x,y)f)(\tau) = \left(\pi(\tau x + y) f\right)(0) = \frac{1}{\sqrt{2 \pi}} \int_{-\infty}^{\infty} e^{i \tau (s \alpha(x) + \beta(x))} e^{i (s \alpha(y) + \beta(y))} \widehat{f}(s) ds.
\end{equation}
Notice that by Lemma \ref{lem:sch_maps_to_smooth} we have that $\fu_f$ is defined by the formula \eqref{eq:fourier_solution} and thus, for every $x,y \in \R^d$ and $\tau \in \R$, we have:
\begin{equation} \label{eq:schwarz_restriction}
\fu_f(\tau x + y) = (\Lambda(x,y)f)(\tau).
\end{equation}

Let us summarize the discussion above in the following theorem:

\begin{thm} \label{thm:schwartz_summary}
Given a vessel $\fV$ satisfying the $VR$ conditions and such that $\sigma_1 > \epsilon I$ for some $\epsilon > 0$, for every initial condition $f \in \tempdist$ we have a weak solution of the system of input compatibility conditions in tempered distributions defined by \eqref{eq:u_as_dist}. Furthermore, the following holds:
\begin{itemize}
\item If $f \in \schwartz$, then $\fu_f$ is a smooth function on $\R^d$, that solves the system of input compatibility equations.

\item Fix $x \in \R^d$, such that $\alpha(x) > \epsilon I$ as well, then for every $y \in \R^d$, we can extend $\Lambda(x,y)$ to an isometric isomorphism from $\Ltwoe$ to $\Ltwoex$. Furthermore, for $x,x^{\prime},y\in \R^d$, the map $\Lambda(x,y) \Lambda(x^{\prime},y)^* $ is a causal isometric isomorphism, i.e., for every $f \in \Ltwoe$ we have the following equalities:
\begin{align} \label{eq:causal_isometry}
\begin{split}
& \int_0^{\infty} \langle \alpha(x) (\Lambda(x,y) f)(s) , (\Lambda(x,y) f)(s) \rangle ds = \int_{0}^{\infty} \langle \alpha(x^{\prime}) (\Lambda(x^{\prime},y)f)(s) , (\Lambda(x^{\prime},y)f)(s) \rangle ds,\\
& \int_{-\infty}^{0} \langle \alpha(x) (\Lambda(x,y) f)(s) , (\Lambda(x,y) f)(s) \rangle ds = \int_{0}^{\infty} \langle \alpha(x^{\prime}) (\Lambda(x^{\prime},y)f)(s) , (\Lambda(x^{\prime},y)f)(s) \rangle ds.
\end{split}
\end{align}
In particular if $y = 0$ and $x^{\prime} = e_1$, then we get that $\Lambda(x,0)$ is a causal isometric isomorphism from $\Ltwoe$ to $\Ltwoex$.

\item If $f$ is a twice continuously differentiable function on $\R$, such that $f,f^{\prime}, f^{\prime\prime} \in \Ltwoe$, then $\fu_f$ is a locally integrable function given by the formula \eqref{eq:fourier_solution} and for every $x,y \in \R^d$, such that $x \in \Pos(\fV)$, the restriction of $\fu_f$ to the line $\tau x + y$ is given by $\Lambda(x,y)f$, namely:
\[
\fu_f(\tau x + y) = (\Lambda(x,y)f)(\tau).
\]
\item If $f$ is a twice continuously differentiable function on $\R$, such that $f,f^{\prime}, f^{\prime\prime} \in \Ltwoe$, then $\fu_f$ is a $C^1$, $\cE$-valued function on $\R^d$ that solves the input compatibility conditions. Furthermore, $\fu_f$ is uniquely determined by its restriction to the $t_1$-axis (or in fact any line with direction vector in $\Pos(\fV)$).

\item If $f \in \Ltwoe$ then for every $\xi \in \Pos(\fV)$ and every $\psi \in \cS(\R^d,\cE)$ we have that:
\[
\langle \fu_f, \psi \rangle = \int_{\xi^{\perp}} \int_{-\infty}^{\infty} \langle (\Lambda(\xi,\eta)f)(s),\psi(\xi s + \eta) \rangle ds d\eta.
\]
\end{itemize}
\end{thm}
\begin{proof}
We have already proved the first claim of the theorem (see Lemma \ref{lem:sch_maps_to_smooth}). To prove the second we note that by virtue of Lemma \ref{lem:sch_maps_to_smooth}, if $f \in \schwartz$ then $\Lambda(x,y) f$ is a smooth function on $\R$ and is a restriction of a solution for the system of input compatibility equations to the line $\tau x + y$ (see \eqref{eq:schwarz_restriction}). Now we recall that $\alpha(x) > 0$ and hence we can apply \cite[Prop.\ 2.1]{BV} ((the proof given there in the case $d=2$ extends verbatim to the case of an arbitrary $d$) to get that the equations \eqref{eq:causal_isometry} hold in this case. Now we can extend $\Lambda(x,y)$ as an isometry from $\Ltwoe$ to $\Ltwoex$.

We will prove the third and the fourth claims together. The proof follows similar lines to \cite[Thm.\ 7.3.5]{Eva10}. We note that by assumption there exists a function $g \in \Ltwoe$ and a constant $C > 0$, such that for every $s \in \R$, we have:
\[
\|\widehat{f}(s)\| \leq C (1 + |s|^2)^{-1}\|\widehat{g}(s)\|.
\]
Hence for every $t \in \R^d$, we have:
\[
\|e^{i \sum_{j=1}^d t_j(s \alpha_j + \beta_j)} \widehat{f}(s)\| \leq C (1+|s|^2)^{-1} \|\widehat{g}(s)\|.
\]
We conclude that:
\begin{multline*}
\int_{-\infty}^{\infty} \|e^{i \sum_{j=1}^d t_j(s \alpha_j + \beta_j)} \widehat{f}(s)\|ds  \leq C \int_{-\infty}^{\infty} (1+|s|^2)^{-1} \|\widehat{g}(s)\| ds \leq  \\ C \sqrt{\int_{-\infty}^{\infty} \|\widehat{g}(s)\|^2 ds} \sqrt{\int_{-\infty}^{\infty} (1+|s|^2)^{-2} ds} < \infty.
\end{multline*}

Hence we have $L_f(t_2,\ldots,t_d)(t_1) = \frac{1}{\sqrt{2\pi}} \int_{-\infty}^{\infty} e^{i \sum_{j=1}^d t_j(s \alpha_j + \beta_j)} \widehat{f}(s)ds$. Thus $\fu_f$ is just integrating $L_f$ against $\varphi$, and therefore we can conclude that $\fu_f$ is a function and we can identify $\fu_f(t_1,\ldots,t_d) = L_f(t_2,\ldots,t_d)(t_1)$. Furthermore, it is now clear that the restriction of $\fu_f$ to lines is given by $\Lambda(x,y) f$.

Let now $0 < |h| < 1$, we write:
\begin{multline*}
\frac{\fu_f(t_1+h,t_2,\ldots,t_d) - \fu_f(t_1,\ldots,t_d)}{h} = \\ \frac{1}{\sqrt{2\pi}h} \int_{-\infty}^{\infty} e^{i \sum_{j=2}^d t_j(s \alpha_j + \beta_j)} \left( e^{i s (t_1+h)} - e^{i s t_1} \right) \widehat{f}(s) ds
\end{multline*}
Since $e^{i s (t_1 + h)} - e^{i s t_1} = i s \int_{t_1}^{t_1 + h} e^{ i s x} dx$, we can use the integral mean value theorem to obtain:
\[
e^{i s (t_1 + h)} - e^{i s t_1} = i s h e^{i s c}.
\]
Here $c$ lies between $t_1$ and $t_1 + h$. Therefore:
\[
\frac{\fu_f(t_1+h,t_2,\ldots,t_d) - \fu_f(t_1,\ldots,t_d)}{h} = \frac{i}{\sqrt{2\pi}} \int_{-\infty}^{\infty} e^{i \sum_{j=2}^d t_j(s \alpha_j + \beta_j)} s e^{i s c} \widehat{f}(s) ds
\]
The integral thus converges, since $\|s \widehat{f}(s)\| \leq C |s|(1 + |s|^2)^{-1} \|\widehat{g}(s)\|$ and the same argument as above applies. Now applying the dominated convergence theorem we can deduce that $\pd{\fu_f}{t_1}$ exists and is continuous. A similar argument applies to every $t_j$.

Uniqueness follows from the fact that we can consider the lines parallel to the $t_1$-axis and use the second part to note that if a solution is $0$ on the $t_1$-axis then using the isometry it is $0$ along every such line and thus it is identically zero.

If $f$ is a Schwartz function, then $\fu_f$ is a smooth function and we have the desired equality by the definition of $\fu_f$ and Fubini's theorem. Now we approximate $f \in \Ltwoe$ be a sequence of Schwartz functions and since $\Lambda(\xi,\eta)$ is an isometry from $\Ltwoe$ to $L^2(\R,\cE,\alpha(\xi))$, we get that $\|\Lambda(\xi,\eta) f\| \leq C(\xi) \|f\|$. Thus for $M$ sufficiently large:
\begin{multline*}
\| \int_{\xi^{\perp}} \int_{-\infty}^{\infty} \langle (\Lambda(\xi,\eta)(f - f_n))(s),\psi(\xi s + \eta) \rangle ds d\eta\| \leq C(\xi) \int_{\xi^{\perp}} \|f - f_n\| \|\psi(\xi s + \eta\|_{L^2} d\eta \\ = C(\xi) \int_{\xi^{\perp}} \|f - f_n\| (1 + \|\eta\|^{2M})^{-1} \| (1 + \|\eta\|^{2M}) \psi(\xi s + \eta) \|_{L^2}d\eta \leq C^{\prime}(\xi) \|f - f_n\| .
\end{multline*}
The last inequality follows from the fact that:
\begin{multline*}
\|(1 + \|\eta\|^{2M} \psi(\xi s + \eta)\|_{L^2}^2 = \int_{-\infty}^{\infty} (1 + s^2)^{-1} \|(1 + \|\eta\|^{2M}) (1 + s^2) \psi(\xi s + \eta)\|^2 ds \\ \leq \sup_{(s,\eta) \in \R^d} \|(1 + \|\eta\|^{2M}) (1 + s^2) \psi(\xi s + \eta)\|^2 \int_{-\infty}^{\infty} (1 + s^2)^{-1} ds.
\end{multline*}
Now applying the same consideration to the case when $\xi = e_1$ we get that for every $\psi \in \cS(\R^d,\cE)$ we have $\langle \fu_{f_n}, \psi \rangle \to \langle \fu_f, \psi \rangle$ and additionally:
\begin{multline*}
\langle \fu_f, \psi \rangle = \lim_{n \to \infty} \langle \fu_{f_n} ,\psi \rangle = \lim_{n \to \infty}\int_{\xi^{\perp}} \int_{-\infty}^{\infty} \langle (\Lambda(\xi,\eta)f_n)(s),\psi(\xi s + \eta) \rangle ds d\eta \\ =  \int_{\xi^{\perp}} \int_{-\infty}^{\infty} \langle (\Lambda(\xi,\eta)f)(s),\psi(\xi s + \eta) \rangle ds d\eta
\end{multline*}
\end{proof}

\begin{rem}
In fact, the third and fourth statements of the preceding theorem is true for functions in the Sobolev space $W^{2,2}(\R,\cE)$. For more details on Sobolev spaces of Banach space valued functions, see \cite{Ama95}.
\end{rem}

\begin{rem} \label{rem:u_intertwines}
Note that for every $f \in \schwartz$ it is immediate from \eqref{eq:fourier_on_line} and the definition of $\pi$ that for every $t \in \R^d$:
\[
\Lambda(x,y)(\pi(t) f) = \Lambda(x,y+t)(f)
\]
Thus, in particular, if $x = e_j$ one of the vectors in the standard basis of $\R^d$, then:
\[
(\Lambda(e_j,0)(\pi(t_j e_j) f))(\tau) = (\Lambda(e_j,t_j e_j)f)(\tau) = (\Lambda(e_j,0)f)(\tau + t_j).
\]
The last equality follows from Equation \eqref{eq:fourier_on_line}. We conclude that if $\alpha_j > \epsilon I$ for some $\epsilon > 0$, then $\Lambda_j = \Lambda(e_j,0)$ intertwines the action of $\R$ on $\Ltwoej$ by translations with the action of $\R$ on $\Ltwoe$ by the restriction of $\pi$ to the one parameter subgroup generated by $e_j$. 
\end{rem}

\section{Unitary Dilation of Semigroups} \label{sec:main}

In order to apply the results of the previous section to dilation theory we need a definition

\begin{dfn} \label{dfn:dissipative_embedding}
Let $\cA = (A_1,\ldots,A_d)$ be a $d$-tuple of commuting dissipative operators on a separable Hilbert space $\cH$. We say that $\cA$ has the dissipative embedding property if $\cA$ can be embedded in a vessel $\fV$ satisfying the $VR$ conditions and such that $\Pos(\fV) \neq \emptyset$ and for every $j=1,\ldots,d$ the standard basis vector $e_j \in \overline{\Pos(\fV)}$.
\end{dfn}

\begin{lem} \label{lem:positive_orthant}
Assume that $\cA$ is a $d$-tuple of commuting dissipative operators on a separable Hilbert space $\cH$ and assume that we can embed $\cA$ in a vessel $\fV$, such that for every $j = 1,\ldots,d$, we have $\sigma_j \geq 0$. Then $\sum_{j=1}^d \im \sigma_j = \cE$ if and only if $\Pos(\fV)$ is not empty. In that case we have that $\R_{>0}^d \subseteq \Pos(\fV)$. In particular if we can embed $\cA$ in a vessel $\fV$, such that for every $j = 1,\ldots,d$, we have $\sigma_j \geq 0$, $\sum_{j=1}^d \im \sigma_j = \cE$ and the vessel satisfies the $VR$ conditions, then $\cA$ has the dissipative embedding property.
\end{lem}
\begin{proof}
If there is a point $\xi \in \Pos(\fV)$, then there exists $\epsilon > 0$, such that for every $u \in \cE$, we have:
\[
\sum_{j=1}^n \xi_j \langle \sigma_j u , u \rangle \geq \epsilon \| u \|^2.
\]
Since each $\sigma_j$ is positive semi-definite, if we omit the terms with $\xi_j \leq 0$ from the above sum, then we just increase it. Hence:
\[
\sum_{\xi_j > 0} \xi_j \langle \sigma_j u , u \rangle \geq \epsilon \| u \|^2.
\]
This implies that the operator $\sum_{\xi_j > 0} \xi_j \sigma_j \geq \epsilon I$ and is in particular invertible, hence $\sum_{j=1}^d \im \sigma_j = \cE$.

Now assume conversely that $\sum_{j=1}^d \im \sigma_j = \cE$. Since the $\sigma_j$ are positive semi-definite, they admit selfadjoint square roots. We have that for every $j = 1,\ldots,d$, $\im \sigma_j \subseteq \im \sqrt{\sigma_j} $ and hence $\sum_{j=1}^d \im \sqrt{\sigma_j} = \cE$. We conclude that the row $(\sqrt{\sigma_1},\ldots,\sqrt{\sigma_d})$ is strictly surjective, therefore there exists $\epsilon > 0$, such that for every $u \in \cE$:
\[
\left\| \begin{pmatrix}
\sqrt{\sigma_1} u \\ \vdots \\ \sqrt{\sigma_d} u
\end{pmatrix}\right\|^2 \geq \epsilon \|u\|^2.
\]
But since $\| \sqrt{\sigma_j} u \|^2 = \langle \sigma_j u , u \rangle$, we have that the point $(1,\ldots,1) \in \Pos(\fV)$.

From the above discussion we have that $\Pos(\fV) \cap \R_{>0}^d \neq \emptyset$. On the other hand for every $\xi \in \R_{>0}^d$, we replace $\sigma_j$ with $\sqrt{\xi_j} \sigma_j$ in the above argument and get that the point $\xi \in \Pos(\fV)$. 
\end{proof}

\begin{thm} \label{thm:dissipative_embedding_dilation}
If a $d$-tuple $\cA$ of commuting dissipative operators on a separable Hilbert space $\cH$ has the dissipative embedding property, then the semigroup of contractions generated by $\cA$ admits a commutative unitary dilation.
\end{thm}

This theorem is a corollary of the following slightly more general theorem, that we will prove in the next section.

\begin{thm} \label{thm:pos_dilation}
Assume that $\cA$ is a $d$-tuple of commuting dissipative operators on a Hilbert space and assume that $\cA$ can be embedded into a vessel $\fV$ that satisfies the $VR$ conditions and $\Pos(\fV) \neq \emptyset$, then the semigroup of contractions generated by $\cA$ restricted to $\overline{\Pos(\fV)}$ admits a commutative unitary dilation.
\end{thm}

\begin{proof}[Proof of Theorem \ref{thm:dissipative_embedding_dilation}] We have that $\R^d_{\geq 0} \subset \overline{\Pos(\fV)}$ (in fact by Lemma \ref{lem:positive_orthant} $\R^d_{>0} \subset \Pos(\fV)$) and thus our semigroup admits a unitary dilation by Theorem \ref{thm:pos_dilation}.
\end{proof}

We shall now deduce a few corollaries from Theorem \ref{thm:dissipative_embedding_dilation}. We first note that we have the following weak form of Ando's theorem (see also \cite{Pta85}, \cite{Slo74} and \cite{Slo82}):

\begin{cor} \label{cor:weak_ando}
Let $A_1$ and $A_2$ be two commuting dissipative operators, such that $\im(A_1-A_1^*) + \im(A_2 - A_2^*)$ is closed. Then the semigroup they generate admits a commutative unitary dilation
\end{cor}
\begin{proof}
Note that every vessel of a pair of commuting operators satisfies the $VR$ conditions vacuously. Furthermore, by our assumption we can embed $A_1$ and $A_2$ into a strict vessel that satisfies the conditions of Lemma \ref{lem:positive_orthant}. Therefore, they have the dissipative embedding property and we are done.
\end{proof}

Recall that given a strongly continuous one-parameter semigroup $C$ of contractions on a Hilbert space $\cH$, by a theorem of Hille and Yosida, it has a generator, namely $C(t) = e^{i A t}$, where $A$ is a dissipative (generally unbounded) operator on $\cH$. If we apply the Cayley transform to $A$ we obtain a contractive operator $T = (A -i I)(A + i I)^{-1}$, that is called the cogenerator of the semigroup. Note that the semigroup can be recovered from $T$ via exponentiation of the inverse Cayley transform, namely $C(t) = \exp(t (T-I)(T+1)^{-1})$. We can also recover the cogenerator from the semigroup directly by the following formula (see \cite{NFBK10} for details):
\begin{equation} \label{eq:cogen}
T = \lim_{s \to 0+} \varphi_s(C(s)),
\end{equation}
\begin{equation} \label{eq:recover_cogen}
\varphi_s(\lambda) =\frac{\lambda - 1 + s}{\lambda - 1 -s} = \frac{1-s}{1+s} - \frac{2s}{1+s} \sum_{n=1}^{\infty} \frac{\lambda^n}{(1+s)^n}.
\end{equation}
Furthermore, it was proved by Sz.-Nagy that $T$ is unitary if and only if $C$ is a unitary semigroup. Now if we have a multi-parameter commutative group of unitaries, its generators are strongly commuting selfadjoint operators (in the sense that the associated projection valued measures commute). Therefore, applying the Cayley transform we get a commuting $d$-tuple of unitaries. Using \eqref{eq:cogen} and \eqref{eq:recover_cogen} we conclude that for a commutative semigroup of contractions that admit a commutative unitary dilation, the cogenerators of the unitary group are  commuting dilations of the cogenerators of the original semigroup.  This discussion leads us to the following negative result:


\begin{prop} \label{prop:parrot}
Not every $d$-tuple of commuting dissipative operators on a separable Hilbert space has the dissipative embedding property.
\end{prop}
\begin{proof}
Consider the Parrot example described in \cite{Par70}. The example is three commuting contractive matrices with spectrum concentrated at $0$.  Hence we can apply the Cayley transform to obtain three commuting dissipative operators, $\cA = (A_1,A_2,A_3)$. If $\cA$ had the dissipative embedding property then by Theorem \ref{thm:dissipative_embedding_dilation} the semigroup they generate would have had a commutative unitary dilation. Thus the cogenerators of this dilation would have been commuting unitary dilations of the original Parrot example and that is a contradiction. 
\end{proof}

The following corollary is also well known, see for example \cite{Pta85} and \cite{Sha10}.

\begin{cor} \label{cor:regular_dilation}
If $A_1,\ldots,A_d$ are doubly commuting, dissipative operators on $\cH$, such that $\sum_{j=1}^d \im(A_j - A_j^*)$ is closed, then the semigroup they generate admits a dilation to a commutative semigroup of unitaries.
\end{cor}
\begin{proof}
Note that since $\sum_{j=1}^d \im(A_j - A_j^*)$ is closed, by Lemma \ref{lem:positive_orthant} we have that $\Pos(\fV)$ contains the positive orthant, where $\fV$ is the strict vessel. By Corollary \ref{cor:doubly_commuting} we have that $\fV$ satisfies the $VR$ conditions and thus we can apply Theorem \ref{thm:dissipative_embedding_dilation} to get the result.
\end{proof}

\section{Construction of the Dilation (Proof of Theorem \ref{thm:pos_dilation})} \label{sec:dilations}

We are given a $d$-tuple of dissipative operators $\cA = (A_1,\ldots,A_d)$ on a separable Hilbert space $\cH$ embedded in a commutative vessel $\fV$ that satisfies the $VR$ conditions and such that $\Pos(\fV) \neq \emptyset$. We will construct a Hilbert space $\cK$, an isometric embedding $\iota \colon \cH \hookrightarrow \cK$ and a unitary representation $\rho \colon \R^d \to B(\cK)$, such that for every $t \in \Pos(\fV)$ and every $h \in \cH$ we have $\iota^* \rho(t) \iota(h) = e^{i t A} h$. Notice that by passing to the SOT-limit we see that the result still holds for $t \in \overline{\Pos(\fV)}$. 

We assume without loss of generality that $e_1 \in \Pos(\fV)$. As in Section \ref{sec:cont_solve} we may (changing the inner product on $\cE$) assume that $\sigma_1 = I_{\cE}$. Recall from \cite{LKMV} and \cite[Prop.\ 1.3.1]{SV14a} that given a $u \in C^1(\R^d,\cE)$ that satisfies the input compatibility conditions, we can solve the time domain system of equations for any initial condition $x(0) = h \in \cH$ using formula \eqref{eq:mult-dim_x}. In particular:
\begin{equation} \label{eq:one-dim_x}
x(t) = x(t,0,\ldots,0) = e^{ i t A_1} \left( h - i \int_0^t e^{-i s A_1} \Phi^* u(s) ds \right).
\end{equation}
This idea allows us to decompose the space of "nice" trajectories of the associated system into a direct sum of the form $\cW_{out} \oplus \cH \oplus \cW_{in}$. Here $\cH$ represents the initial condition. We then introduce using the theory developed in Section \ref{sec:cont_solve} a unitary representation of $\R^d$ on this space, such that the compression of its $\Pos(\fV)$ semigroup to $\cH$ is our initial semigroup of contractions restricted to $\Pos(\fV)$.

Let us consider first the case of a single operator with $\sigma_1 = I_{\cE}$. This is a classical construction one can find for example in \cite{CBMS87}, \cite{LaxPhl89}, \cite{NFBK10}, \cite{Pav99} and \cite{Sch55}. Note that in this case the one-parameter semigroup $T(t) = e^{i t A_1}$ for $t > 0$ is a semigroup of contractions on $\cH$. Set $\cK = L^2(\R_{<0},\cE) \oplus \cH \oplus L^2(\R_{>0},\cE)$ and we are going to describe a dilation of $T$ to $\cK$. To do this we need the following lemma:
\begin{lem} \label{lem:one_dim_ext}
Given a triple $(y,h,u) \in \cK$ there exists a unique (strongly) absolutely continuous function $x \colon \R \to \cH$, such that for $t > 0$:
\[
i x^{\prime}(t) + A_1 x(t) = \Phi^* u(t)
\]
and for $t < 0$:
\[
i x^{\prime}(t) + A_1^* x(t) = \Phi^* y(t).
\]
We then extend $y$ to $\tilde{y} \colon \R \to \cE$ by defining for $t> 0$:
\[
\tilde{y}(t) = u(t) - i \Phi x(t)
\]
and we extend $u$ to $\tilde{u} \colon \R \to \cE$ by defining for for $t < 0$:
\[
\tilde{u}(t) = y(t) + i \Phi x(t).
\]
Then we have $\tilde{y},\tilde{u} \in L^2(\R,\cE)$.
\end{lem}
\begin{proof}
We define $x$ in terms of $u$ using \eqref{eq:one-dim_x} for $t > 0$ and in terms of $y$ using the analog of \eqref{eq:one-dim_x} for the adjoint system (see  \eqref{eq:adjoint_system}) for $t < 0$:
\begin{equation} \label{eq:explicit_x}
x(t) = \begin{cases} e^{ i t A_1}\left( h - i \int_0^{t} e^{-i s A_1} \Phi^* u(s) ds\right), & t > 0;\\
e^{i t A_1^*} \left( h + i \int_t^{0} e^{-i s A_1^*} \Phi^* y(s) ds\right), & t < 0; \\
\end{cases}
\end{equation}

 Then clearly $x$ is an absolutely continuous $\cH$-valued function on $\R$. Now from the energy conservation equations \eqref{eq:energy_balance} we get that for $t > 0$:
\[
\|x(t)\|^2 - \|h\|^2 = \int_0^t \langle \tilde{u}, \tilde{u} \rangle - \int_0^t \langle \tilde{y}, \tilde{y} \rangle.
\]
We conclude that:
\[
\int_0^t \langle \tilde{y}, \tilde{y} \rangle \leq \int_0^t \langle \tilde{u}, \tilde{u} \rangle + \|h\|^2.
\]
Therefore $\tilde{y} \in L^2(\R,\cE)$ and similarly for $\tilde{u}$.
\end{proof}
Notice that the trajectory $(\tilde{y},x,\tilde{u})$ is the unique trajectory of the system (equivalently $(\tilde{u},x,\tilde{y})$ is a unique trajectory of the adjoint system), such that $\tilde{y}|_{\R_{<0}} = y$, $x(0) = h$ and $\tilde{u}|_{\R_{>0}} = u$ .
The following proposition provides a dilation of the one parameter semigroup of contractions generated by $A_1$.
\begin{prop} \label{prop:one_dim_dilation}
There exists a unitary representation $\rho$ of the Lie group $\R$ on $\cK$, such that if $P$ is the projection onto $\cH$, then:
\[
P \rho(t) (0,h,0) = (0, e^{i t A_1} h, 0).
\]
\end{prop}
\begin{proof}
Let $(y,h,u) \in \cK$, let $(\tilde{y},x,\tilde{u})$ be the unique trajectory of the system associated to our triple, where $x$ is defined by \eqref{eq:explicit_x}. Denote by $\tilde{y}_t(s) = \tilde{y}(s+t)$ and similarly $\tilde{u}_t(s) = \tilde{u}(s+t)$, for every $t \in \R$. Now we define our representation as follows:
\[
\rho(t)(y,h,u) = (\tilde{y}_t|_{\R_{<0}},x(t),\tilde{u}_t|_{\R_{>0}}).
\]
Using the energy balance equations we obtain for $t > 0$:
\[
\|x(t)\|^2 = \|x(0)\|^2 + \int_0^t \langle u(s) ,u(s) \rangle ds - \int0^t \langle \tilde{y}(s), \tilde{y}(s) \rangle ds. 
\]
Therefore for $t > 0$:
\begin{multline*}
\|\rho(t)(y,h,u)\|^2 = \|x(t)\|^2 + \int_{-\infty}^0 \langle \tilde{y}_t(s), \tilde{y}_t(s) \rangle ds + \int_0^{\infty} \langle \tilde{u}_t(s), \tilde{u}_t(s) \rangle ds = \\ \|x(t)\|^2 + \int_{-\infty}^0 \langle \tilde{y}(s+t), \tilde{y}(s+t) \rangle ds + \int_0^{\infty} \langle \tilde{u}(s+t), \tilde{u}(s+t) \rangle ds = \\ \|h\|^2 + \int_0^t \langle u(s) ,u(s) \rangle ds - \int_0^t \langle \tilde{y}(s), \tilde{y}(s) \rangle ds +  \int_{-\infty}^{t} \langle \tilde{y}(s), \tilde{y}(s) \rangle ds + \\ \int_{t}^{\infty} \langle \tilde{u}(s), \tilde{u}(s) \rangle ds = \|h\|^2 + \int_0^{\infty} \langle u(s), u(s) \rangle ds + \int_{-\infty}^0 \langle y(s), y(s) \rangle ds = \|(y,h,u)\|^2
\end{multline*}

Hence $\rho$ is a unitary representation of $\R$. Now note that from \eqref{eq:explicit_x} we have that for the triple $(0,h,0)$ the associated $x$ is:
\[
x(t) = \begin{cases} e^{i t A_1} h, & t> 0\\ e^{i t A_1^*} h, & t < 0 \end{cases}.
\]
Hence for positive $t$ we obtain that $P\rho(t)(0,h,0) = (0,x(t),0) = (0,e^{ i t A_1} h,0)$.
\end{proof}

This idea leads us to consider the following construction. We construct weak solutions of the system of input and output compatibility equations from $\tilde{u}$ and $\tilde{y}$ and we plug these weak solution, more precisely the functions $\Lambda(\xi,\eta)(\tilde{u})$ and $\Lambda(\xi,\eta)(\tilde{y})$, into the formula \eqref{eq:mult-dim_x} to get a state function $x$ on $\R^d$ that is absolutely continuous on lines $\xi t + \eta$, such that $\xi \in \Pos(\fV)$. However, first we need a dense subspace to work with.

\begin{lem} \label{lem:k_0_dense}
Let $\cK_0 \subset \cK$ be the subspace of triples $(y,h,u)$, such that both $\tilde{y}$ and $\tilde{u}$ are twice continuously differentiable and both of the derivatives are square-summable, then $\overline{\cK_0} = \cK$.
\end{lem}
\begin{proof}
First note that since both $\tilde{u}$ and $\tilde{y}$ are twice continuously differentiable, we have that $x$ is thrice continuously differentiable and $h = x(0)$. Using \eqref{eq:explicit_x} we get that $h$ is independent of $u$ and $y$ (it is in fact the initial condition). We must require that $\lim_{t \to 0+} u(t)$ and $\lim_{t \to 0-} y(t)$ exist and we denote them by $\tilde{u}(0)$ and $\tilde{y}(0)$, respectively. Similarly for their derivatives. Furthermore, we have the following condition on the values at $0$ and the derivatives:
\[
\tilde{u}(0) - \tilde{y}(0) = i \Phi h, \quad \tilde{u}^{\prime}(0) - \tilde{y}^{\prime}(0) = \Phi \Phi^* \tilde{u}(0) -\Phi A_1 h = \Phi \Phi^* \tilde{y}(0) - \Phi A_1^* h.
\]
\[
\tilde{u}^{\prime\prime}(0) - \tilde{y}^{\prime\prime}(0) = \Phi \Phi^* \sigma u^{\prime}(0) - i A_1^2 h + i A_1 \Phi^* \sigma u(0).
\]
Therefore, a choice of $h \in \cH$ forces three conditions on both $u$ and $y$. However, a standard argument shows that twice continuously differentiable functions, with boundary conditions on them and their derivatives are dense in $L^2(\R_{+},\cE)$ and $L^2(\R_{-},\cE)$.
\end{proof}

This lemma allows us to define for every $t \in \R^d$ an operator on $\cK_0$. Let us denote by $\fy_f$ the weak solution for the system of output compatibility equations with the initial condition $f$ on the $t_1$-axis and by $\Lambda_*(x,y)f$ the associated linear map. Given $(y,h,u) \in \cK_0$, we construct $\tilde{u}$ and $\tilde{y}$. We then apply the transform $\fu$ and $\fy$, respectively, to get (by Theorem \ref{thm:schwartz_summary}) continuously differentiable functions $u^{\dagger} = \fu_{\tilde{u}}$ and $y^{\dagger} = \fy_{\tilde{u}}$ on $\R^d$, that solve the system of input and output compatibility conditions.
Since these functions are continuously differentiable we can solve the associated system of our vessel with initial condition $h$, using \eqref{eq:mult-dim_x} to obtain a twice continuously differentiable function $x$. Using the second equation of the system we obtain an output function $z$, that solves the system of output compatibility equations and coincides with $\tilde{y}$ on the $t_1$-axis, thus by uniqueness $z = y^{\dagger}$. We define:
\[
\rho(t)(y,h,u) = (y^{\dagger}(\cdot,t_2,\ldots,t_d)|_{<t_1},x(t),u^{\dagger}(\cdot,t_2,\ldots,t_d)|_{> t_1}).
\]
Note that it is immediate that $\rho(0)$ is the identity on $\cK_0$.

\begin{rem}
Note that it is possible to use Equation \eqref{eq:mult-dim_x} to construct the state signal for all $t \in \R^d$, since the operators $A_1,\ldots,A_d$ are bounded and thus generate a group. In case these operators were unbounded one could run the adjoint system first to go back and then run the original system to obtain the value of the state signal. One of course would have in that case to show the commutation of the two actions. 
\end{rem}

\begin{lem} \label{lem:rep_is_isometry}
For every $(y,h,u) \in \cK_0$ we have that:
\begin{itemize}
\item For every $t,s \in \R^d$, we have that $\rho(t)\rho(s)(y,h,u) = \rho(t+s)(y,h,u)$. Therefore, $\rho$ is in fact a representation of $\R^d$.

\item For every $t \in \Pos(\fV)$, we have that $\|\rho(t)(y,h,u)\| = \|(y,h,u)\|$ and hence, in particular $\rho(t)$ extends to an isometry on $\cK$.

\item For every $t \in -\Pos(\fV)$, we have that $\|\rho(t) (y,h,u)\| = \|(y,h,u)\|$ and hence $\rho(t)$ is in fact a unitary on $\cK$, for every $t \in \Pos(\fV) \cup - \Pos(\fV)$.

\end{itemize}
\end{lem}
\begin{proof}
The first claim follows from the uniqueness part of Theorem \ref{thm:schwartz_summary}. Namely, since $u^{\dagger}$ and $y^{\dagger}$ were determined uniquely by their restriction to any line with direction vector in $\Pos(\fV)$ and additionally restriction commutes with shifts, the claim follows.

For $t \in \R^d$ we write $\Lambda(t) = \Lambda(t,0)$ and $\sigma(t) = \sum_{j=1}^d t_j \sigma_j$. To prove the second claim we apply \eqref{eq:energy_balance} (modified to the straight line segment from $0$ to $t$) to get that for every $t \in \Pos(\fV)$ we have:
\begin{multline*}
\|x(t)\|^2 = \|h\|^2 + \int_0^1 \langle \sigma(t) (\Lambda(t)\tilde{u})(w) , (\Lambda(t)\tilde{u})(w) \rangle dw - \\ \int_0^1 \langle \sigma(t) (\Lambda_*(t)\tilde{y})(w) ,  (\Lambda_*(t)\tilde{y})(w) \rangle dw
\end{multline*}
(note that by Theorem \ref{thm:schwartz_summary} we have that $\Lambda(t)(\tilde{u})(w) = u^{\dagger}(t w)$ and $\Lambda_*(t)(\tilde{y})(w) = y^{\dagger}(t w)$). Let us write $u_t = \Lambda(t)(\tilde{u})$ and $y_t = \Lambda_*(t)(\tilde{y})$. Applying Theorem \ref{thm:schwartz_summary} again we obtain:
\begin{align*}
\begin{split}
& \int_0^{\infty} \langle \sigma(t) u_t(w), u_t(w) \rangle dw = \int_0^{\infty} \langle u(w), u(w) \rangle dw, \\ & \int_{-\infty}^0 \langle \sigma(t) y_t(w), y_t(w) \rangle dw = \int_{-\infty}^0 \langle y(w), y(w) \rangle dw.
\end{split}
\end{align*}
We now compute:
\begin{multline*}
\|(y,h,u)\|^2 = \|h\|^2 + \int_0^{\infty} \langle u(w), u(w) \rangle dw + \int_{-\infty}^0 \langle y(w), y(w) \rangle dw = \\ \|x(t)\|^2 -  \int_0^1 \langle \sigma(t) u_t(w) , u_t(w) \rangle dw + \int_0^1 \langle \sigma(t) y_t(w) ,  y_t(w) \rangle dw + \\ \int_0^{\infty} \langle \sigma(t) u_t(w), u_t(w) \rangle dw + \int_{-\infty}^0 \langle \sigma(t) y_t(w), y_t(w) \rangle dw  = \\ \|x(t)\|^2+ \int_1^{\infty} \langle \sigma(t) u_t(w) , u_t(w) \rangle dw + \int_{-\infty}^1 \langle \sigma(t) y_t(w), y_t(w) \rangle dw = \\
\|x(t)\|^2 + \int_0^{\infty} \langle \sigma(t) (\Lambda(t,t) \tilde{u})(w), (\Lambda(t,t) \tilde{u})(w) \rangle dw + \\ \int_{-\infty}^0 \langle \sigma(t) (\Lambda_*(t,t) \tilde{y})(w), (\Lambda_*(t,t) \tilde{y})(w) \rangle dw.
\end{multline*}
Another application of Theorem \ref{thm:schwartz_summary} gives us that both $\Lambda(t,t) \Lambda(e_1,t)^*$ and $\Lambda_*(t,t) \Lambda_*(e_1,t)^*$ are causal isometric isomorphisms and hence:
\[
\|(y,h,u)\| = \|\rho(t)(y,h,u)\|.
\]

The third claim is identical to the second but we exchange the roles of $u$ and $y$. Since for $t\in \Pos(\fV) \cup -\Pos(\fV)$ we have that $\rho(t)\rho(-t) = \rho(-t)\rho(t) = 1$ by the first part of the lemma, we conclude that $\rho(t)$ is a surjective isometry and hence a unitary and that $\rho(-t) = \rho(t)^*$.

\end{proof}

\begin{lem} \label{lem:weak_trajectory}
For every $(y,h,u) \in \cK$, the following function is well defined for every $t \in \R^d$, $x(t) = P_{\cH} \rho(t)(y,h,u)$. Furthermore, $x$ is continuous and absolutely continuous on lines with a direction vector $\xi \in \Pos(\fV)$ and we have the equality:
\begin{equation} \label{eq:general_x_on_lines}
x(\xi s + \eta) = e^{i (\xi s +\eta) A}\left( h + i \int_0^s e^{- i (\xi w + \eta) A} \Phi^* \sigma(\xi) (\Lambda(\xi)\tilde{u})(w) dw \right).
\end{equation}
And thus for every line $\xi s + \eta$, with $\xi \in \Pos(\fV)$, we have:
\[
\dfrac{dx}{ds} = i \xi A x - i \Phi^* \sigma(\xi) (\Lambda(\xi,\eta) \tilde{u}).
\]
\end{lem}
\begin{proof}
For every $\epsilon > 0$ we choose $(y_0,h,u_0) \in \cK_0$, such that $\|(y,h,u) - (y_0,h,u_0)\| < \epsilon$. Then for every $t \in \R^d$ we get $\|x_{0}(t) - P_{\cH}\rho(t)(y,h,u)\| < \epsilon$, since $P_{\cH} \rho(t)$ is a contraction. Since $x_0$ is continuous a standard $\epsilon/3$ argument shows that $x$ is continuous.

If we prove that $x(\xi s + \eta)$ has the form described in \eqref{eq:general_x_on_lines}, then we immediately see that $x$ is absolutely continuous on those lines. Let us fix $\xi \in \Pos(\fV)$, then:
\[
 x_0(\xi s) = e^{i \xi A} h + i\int_{0}^{s} e^{i (\xi - w \xi) A} \Phi^* \sigma(\xi) (\Lambda(\xi)(\tilde{u_0})(w) dw.
\]
We note that for every $0 \leq w \leq 1$, we have $1 - w \geq 0$ and thus $e^{i (\xi - w \xi) A}$ is a contraction. Hence:
\begin{multline*}
x_0(\xi s) - x(\xi s) = \left\| \int_{0}^{s} e^{i (\xi - w \xi) A} \Phi^* \sigma(\xi) (\Lambda(\xi)(\tilde{u_0} - \tilde{u})(w) dw \right\| \leq \\ C  \int_0^1 \| \sigma(\xi) (\Lambda(\xi)(\tilde{u_0} - \tilde{u})(w) \| dw \leq C \sqrt{\int_0^1 \| \sigma_t (\Lambda(\xi)(\tilde{u_0} - \tilde{u})(w)\|^2 dw}.
\end{multline*}
Now using functional calculus we note that there is a constant $C^{\prime}$, such that for every $\xi \in \cE$ the inequality $\|\sigma(t) \xi\|^2 \leq C^{\prime} \langle \sigma(t) \xi, \xi \rangle$ holds. We thus conclude that:
\begin{multline*}
\left\| \int_{0}^{s} e^{i (\xi - w \xi) A} \Phi^* \sigma(\xi) (\Lambda(\xi)(\tilde{u_0} - \tilde{u})(w) dw \right\| \\ \leq C \sqrt{C^{\prime}} \sqrt{\int_0^{\infty} \langle \sigma(\xi) (\Lambda(\xi)(\tilde{u_0} - \tilde{u})(w), (\Lambda(\xi)(\tilde{u_0} - \tilde{u})(w) \rangle dw} < C \sqrt{C^{\prime}} \sqrt{\epsilon}.
\end{multline*}
Letting $\epsilon$ tend to $0$ we get the desired result.

\end{proof}

\begin{rem}
For every line $\xi s + \eta$, with $\xi \in \Pos(\fV)$, we have that:
\[
(\Lambda(\xi,\eta)\tilde{y})(s) = (\Lambda(\xi,\eta)(\tilde{u})(s) + i\Phi x(\xi s + \eta).
\]
\end{rem}

\begin{proof}[Proof of Theorem \ref{thm:dissipative_embedding_dilation}]
First we note that by Lemma \ref{lem:rep_is_isometry} and the fact that $\Pos(\fV)$ spans $\R^d$ we get that $\rho$ is a unitary representation of $\R^d$ on $\cK$.

Now we need to check that if $P_{\cH} \colon \cK \to \cH$ is the orthogonal projection, then $P_{\cH} \rho(t) (0,h,0) = e^{i t A}h$, for every $h \in \cH$. This, however, follows immediately from Lemma \ref{lem:weak_trajectory}.
\end{proof}

Lastly, we discuss the minimality of the unitary dilation that we have constructed. Recall that the dilation $\rho$ is minimal if $\overline{\Span\{\rho(t)\cH \mid t \in \R^d\}} = \cK$ or equivalently that there exists no $\rho$-invariant subspace of $\cH^{\perp}$. For simplicity we shall assume that $A_1,\ldots,A_d$ have the dissipative embedding property.

\begin{lem} \label{lem:injective}
Let $\cX$ and $\cY$ be Hilbert spaces and $A \colon \cX \to \cY$ be an injective bounded linear operator. Then the induced linear map $A \colon \cS^{\prime}(\R^d,\cX) \to \cS^{\prime}(\R^d,\cY)$ is injective.
\end{lem}
\begin{proof}
We need to show that for every $\varphi \in \cS^{\prime}(\R^d,\cX)$, such that $A \varphi = 0$ we have that for every $f \in \cS(\R^d,\cX)$, $\langle \varphi, f \rangle = 0$. Since $\cS(\R^d,\cX) \cong \cS(\R^d) \widehat{\otimes} \cX$ if we show the claim for elementary tensors we will be done, since their span is dense. So given a function $f \in \cS(\R^d)$ and $v \in \cX$, we consider the function $f(t) v$. If $v \in \im A^*$, then $v = A^* w$ and we get:
\[
\langle \varphi, f(t) v \rangle = \langle A \varphi , f(t) w \rangle = 0.
\]
Since $\im A^*$ is dense, for every $v \in \cX$, we can choose a sequence $v_n \in \im A^*$, that converges to $v$. Thus the sequence $f(t) v_n$ will converge in $\cS(\R^d,\cX)$ to $f(t) v$. By continuity of $\varphi$ we have that $\langle \varphi, f(t) v \rangle = 0$ for every $v \in \cX$ and we are done.

\end{proof}

\begin{lem} \label{prop:weakly_strict_to_minimal}
The dilation obtained above is minimal if $\fV$ is weakly strict.
\end{lem}
\begin{proof}
Let us consider the subspace $\cL = \overline{\Span\{\rho(t)\cH \mid t \in \R^d\}} \subset \cK$ and assume the vessel $\fV$ is weakly strict. Consider the orthogonal complement $\cL^{\perp}$ of $\cL$. Since $\rho$ is unitary and $\cL$ is $\rho$ invariant, we have that $\cL^{\perp}$ is $\rho$-invariant as well. Note that every vector in $\cL^{\perp}$ is of the form $(y,0,u)$ and by invariance we have that the state function $x$ we generate is identically $0$. From Lemma \ref{lem:weak_trajectory} we get that for $\xi \in \Pos(\fV)$  $\Phi^* \sigma(\xi) u(t) = 0$ almost everywhere on every line in direction $\xi$. By Theorem \ref{thm:schwartz_summary} for every $\psi \in \cS(\R^d,\cE)$ we have:
\[
\langle \Phi^* \sigma(\xi) \fu_{\tilde{u}}, \psi \rangle = \int_{\xi^{\perp}} \int_{-\infty}^{\infty} \langle \Phi^* \sigma(\xi) (\Lambda(\xi,\eta) \tilde{u}(s), \psi(s \xi + \eta) \rangle ds d\eta = 0.
\]
Now we note that since $\fV$ is weakly strict we have that $\cap_{\xi \in \Pos(\fV)} \ker \Phi^* \sigma(\xi) = 0$. Thus if we consider the operator $(\Phi^* \sigma_1,\ldots,\Phi^* \sigma_d) \colon \cE \to \cE^d$, then by Lemma \ref{lem:injective} it is injective and we conclude that $\fu_{\tilde{u}} = 0$. We need now only to deduce that $\tilde{u} = 0$. Let now $g \in \cS(\R,\cE)$ and $h \in \cS(\R^{d-1})$. We write $t^{\prime} = (t_2,\ldots,t_d)$ and define a function $f(t_1,\ldots,t_d) = g(t_1) h(t^{\prime}) \in \cS(\R^d,\cE)$ and we get:
\[
0 = \langle \fu_{\tilde{u}}, f \rangle = \int_{\R^{d-1}} g(t^{\prime}) \langle L_{\tilde{u}}(t^{\prime}), h \rangle d t^{\prime}.
\]
Now we assume that $\tilde{u} \neq 0$, ,then there exists $h \in \cS(\R,\cE)$, such that $\langle \tilde{u}, h \rangle \neq 0$. Since $L_{\tilde{u}}$ is a smooth function there exists a $\delta > 0$, such that for every $t^{\prime} \in \R^{d-1}$ if $\|t^{\prime}\| < \delta$, then without loss of generality $\operatorname{Re}\langle L_{\tilde{u}}(t^{\prime}), h \rangle > 0$. Furthermore, we can choose $g(t^{\prime})$ to be a bump function that is $1$ on the ball of radius $\delta/2$ and is zero outside of a ball of radius $\delta$ and is always positive. Thus:
\[
\operatorname{Re} \int_{\R^{d-1}} g(t^{\prime}) \langle L_{\tilde{u}}(t^{\prime}), h \rangle d t^{\prime} = \int_{B_{\delta}} g(t^{\prime}) \operatorname{Re} \langle L_{\tilde{u}}(t^{\prime}), h \rangle d t^{\prime} > 0. 
\]
This is a contradiction and thus $\tilde{u} = 0$. Since $x$ we can conclude that $y = 0$ and we are done.
\end{proof}

This condition is sufficient, but we can also get a necessary condition. To this end we need the following simple lemma:

\begin{lem} \label{lem:invariant_subspace}
Let $\cE$ be a Hilbert space and $W \subset \cE$ a closed subspace. Let $\alpha_j$ and $\beta_j$, $j=1,\ldots,d$ be operators on $\cE$, such that for every $s \in \C$, and every $1 \leq j < k \leq d$ we have:
\[
[ s \alpha_j + \beta_j, s \alpha_k + \beta_k] = 0.
\]
Then the following are equivalent:
\begin{enumerate}[(i)]
\item For every $j = 1,\ldots,d$ we have that $\alpha_j W \subset W$ and $\beta_j W \subset W$,

\item For every polynomial $p \in \C[z_1,\ldots,z_d]$ and every $s \in \C$, we have that $p(s \alpha_1 + \beta_1, \ldots, s \alpha_d + \beta_d) W \subset W$,

\item For every polynomial $p \in \C[z_1,\ldots,z_d]$ and \emph{almost} every $s \in \C$, we have that $p(s \alpha_1 + \beta_1, \ldots, s \alpha_d + \beta_d) W \subset W$,
\end{enumerate}
\end{lem}
\begin{proof}
The equivalence of $(i)$ and $(ii)$ is obvious as well as the fact that $(ii)$ implies $(iii)$. To see that $(iii)$ implies $(ii)$ note that for every $\xi \in W$ and every polynomial $p \in \C[z_1,\ldots,z_d]$ the following function is a continuous function in $s$:
\[
s \mapsto p(s \alpha_1 + \beta_1, \ldots, s \alpha_d + \beta_d) \xi.
\]
Composing with the projection onto $\cE/W$ we get a continuous function that is $0$ almost everywhere and thus is identically $0$.
\end{proof}

So assume $\fV$ is not weakly strict and write $\cW = \cap_{j=1}^d \ker \Phi^* \sigma_j$. Assume that there exists a vector $w \in \cW$, such that for every $j = 1,\ldots,d$ we have that $\alpha_j w \in \cW$ and $\beta_j w \subset W$. Fix some compact set  $K = [a,b] \subset \R_{>0}$ and set $\tilde{u} = \cF^{-1}(1_{K} w)$. Note that the values of $u$ are all in $\cW$ since $\cW$ is closed. Furthermore, we have that:
\[
\fu(t_1,\ldots,t_d) = \frac{1}{\sqrt{2 \pi}} \int_{-\infty}^{\infty} e^{i \sum_{j=1}^d t_j (s \alpha_j + \beta_j)} \widehat{u}(s) ds = \frac{1}{\sqrt{2 \pi}} \int_a^b e^{i \sum_{j=1}^d t_j (s \alpha_j + \beta_j)} w ds.
\]
Now the last expression belongs to $\cW$ by our assumption. Then if we construct the associated state $x$ for an initial condition $h$ we get that $x =0$ identically. Therefore, the triple $(u|_{\R_{<0}},0,u|_{\R_{>0}})$ is a non-zero vector orthogonal to $\cL$, the space defined in the Lemma above. 

\begin{prop} \label{prop:minimal_then_no_submodule}
If there exists a closed subspace $0 \neq \cM \subset \cW$, invariant under $\alpha_j$ and $\beta_j$ for every $j=1,\ldots,d$, then the construction yields a non-minimal dilation. In other words if we have a minimal dilation then there is no such subspace of $\cW$.
\end{prop}

\bibliographystyle{plain}
\bibliography{Livsic_Commutative_Vessels_of_3_Operators}{}
\end{document}